\documentclass[10pt]{article}

\usepackage{amsmath,amssymb,amsthm}
\usepackage{color}
\usepackage{comment}
\usepackage{epsfig} 
\usepackage{epstopdf}
\usepackage{graphics} 
\usepackage{graphicx}
\usepackage[utf8]{inputenc}
\usepackage{multicol}
\usepackage{paralist}
\usepackage{subfigure}
\usepackage{tikz}
\usetikzlibrary{arrows}
\usetikzlibrary{decorations.markings}
\usetikzlibrary{intersections}

\usepackage{pgfplots}
\pgfplotsset{every axis/.append style={
                    axis x line=middle,    % put the x axis in the middle
                    axis y line=middle,    % put the y axis in the middle
                    axis line style={-,color=darkgray}, % arrows on the axis
                    xlabel={$a$},          % default put x on x-axis
                    ylabel={$b$},          % default put y on y-axis
                    grid=none
					}}
\usepgfplotslibrary{fillbetween}
            
\usepackage[colorlinks=true]{hyperref}
\hypersetup{urlcolor=blue, citecolor=red}

\definecolor{mycolor}{RGB}{255,214,179}
\newtheorem{theorem}{Theorem}[section]

\newtheorem{lemma}[theorem]{Lemma}
\newtheorem{proposition}{Proposition}
\newtheorem{conjecture}{Conjecture}
\theoremstyle{definition}
\newtheorem{definition}[theorem]{Definition}
\newtheorem{remark}[theorem]{{\sc Remark}}

\newtheorem{mtema}{{\sc Theorem}}

\title{Renormalization of two-dimensional piecewise linear maps: Abundance of 2-D strange attractors}
\date{}
\author{\centerline{\small (Dedicated to the memory of Professor Welington de Melo)} \\[1.5ex] 
\textsc{A. Pumari\~{n}o\thanks{Dep. de Matem\'{a}ticas. Calvo Sotelo s/n, 33007 Univ. de Oviedo, Spain. E-mail address: apv@uniovi.es, jarodriguez@uniovi.es, vigilkike@gmail.com}, J. A. Rodr\'{\i}guez$^{\ast}$, E. Vigil$^{\ast}$}}

\begin{document}
\maketitle

\begin{abstract}
For a two parameter family of two-dimensional piecewise linear maps and for every natural number $n,$ we prove not only the existence of intervals of parameters for which the respective maps are n times renormalizable but also we show the existence of intervals of parameters where the coexistence of at least $2^n$ strange attractors takes place. This family of maps contains the two-dimensional extension of the classical one-dimensional family of tent maps.

2010 AMS Classification : Primary: $37C70$, $37D45$; Secondary: $37D25$. Key words and phrases: Piecewise linear maps, renormalization, strange attractor.
\end{abstract}

% ------------------------------ SECTION ------------------------------

\section{Introduction} \

Several papers in the last century (see \cite{BeCa2},\cite{MoVi},\cite{PuRo1}) have been devoted to analytically prove the existence of strange attractors.

\begin{definition}\label{defstratt}
An \textbf{attractor} for a transformation $ f $ defined on a manifold $ \mathcal{M} $ is a compact, $f-$invariant and transitive set $ \mathcal{A} $ whose stable set
\begin{equation*}
W^{s}( \mathcal{A} ) = \left\{ z \in  \mathcal{M} : \; d(f^{n}(z),\mathcal{A}) \rightarrow 0 \text{ as } n \rightarrow \infty \right\}
\end{equation*} 
has nonempty interior. An attractor is said to be \textbf{strange} if it contains a dense orbit $\left\{ f^{n}(z_{1}): \; n \geq 0 \right \} $ displaying exponential growth of the derivative: there exists some constant $ c > 0 $ such that, for every $ n \geq 0,$
\begin{equation*}
\| Df^{n}(z_{1}) \| \geq e^{cn}.
\end{equation*}
\end{definition}

The above definition, particularly, implies that strange attractors display, in a dense orbit, at least one positive
Lyapounov exponent. The strange attractors found in the previous references are one-dimensional attractors in the sense that they only have one positive Lyapounov exponent. The proof on the existence of these attractors is strongly based on the unfolding of a limit family of unimodal maps. In \cite{BeCa2} and \cite{MoVi} this family is the quadratic one $ \{ f_{a} \}_{a \in [1,2]} $ defined by 
\begin{equation*} 
f_a : x \in [-1,1] \mapsto f_{a} (x)=1-ax^{2} .
\end{equation*}
This family has strange attractors for values of the parameters in a set of positive Lebesgue measure, see \cite{BeCa1}. In a simpler way, the existence of strange attractors can be proved for the family of tent maps $ \{ \lambda_{\mu} \}_{\mu\in (1,2]} $ given by
\begin{equation*}
\lambda_{\mu} : x \in [0,2] \mapsto \lambda_{\mu}(x)=
\left\{
\begin{array}{ll}
\mu x & \mbox{, if } x \in [0,1] \\
\mu (2-x) & \mbox{, if } x \in [1,2] .
\end{array}
\right.
\end{equation*}

In this case, the interval $ I_{\mu}= [\mu(2-\mu),\mu]$ is invariant by $\lambda_{\mu}$ for every $ \mu \in (1,2] $ and it is a strange attractor for every $ \mu \in (\sqrt{2},2].$ Strange attractors with several pieces may also be obtained for values of the parameter in $ (1,\sqrt{2}]$ by using renormalization techniques (see \cite{BrBr}, \cite{MeSt} or \cite{Mel} for details). In order to get abundance of strange attractors with two positive Lyapounov exponents, the authors consider in \cite{PuRoTaVi0} a generic two-parameter family $ f_{a,b} : M \longrightarrow M $ of three-dimensional diffeomorphisms unfolding a generalized homoclinic tangency, as it was originally defined in \cite{Tat}. The unstable manifold involved in this homoclinic tangency has dimension two and the limit family is conjugate to the family of two-dimensional endomorphisms defined on $\mathbb{R}^{2}$ by 
\begin{equation*}
T_{a,b} (x,y)=(a+y^{2},x+by).
\end{equation*}

The dynamical behaviour of the  family $T_{a,b}$ is rather complicated as was numerically pointed out in \cite{PuTa2} and, in particular, the attractors exhibited by $T_{a,b}$ for a large set of parameters seem to be two-dimensional strange attractors. Moreover, in \cite{PuTa1}, a curve of parameters
\begin{equation}\label{curve}
G =\left\{ (a(s),b(s))= \left(-\frac{1}{4}s^{3}(s^{3}-2s^{2}+2s-2),-s^{2}+s\right) : \; s \in \mathbb{R}^2 \right\}
\end{equation} 
has been constructed in such a way that the respective transformation $T_{a(s),b(s)}$ has, for every $ s \in [0,2],$ an invariant region in $\mathbb{R}^{2}$ homeomorphic to a triangle. This curve contains the point $(-4,-2)$ and the map $T_{-4,-2}$ is conjugate to the non-invertible piecewise affine map
\begin{equation*}
\Lambda(x,y)=
\left\{
\begin{array}{ll}
(x+y,x-y) & \mbox{, if } (x,y)\in \mathcal{T}_{0} \\
(2-x+y,2-x-y)& \mbox{, if } (x,y) \in \mathcal{T}_{1},
\end{array}
\right.
\end{equation*}
defined on the triangle $\mathcal{T}=\mathcal{T}_{0}\cup\mathcal{T}_{1}$, where
\begin{eqnarray*} 
\mathcal{T}_{0} & = & \left\{(x,y): \; 0\leq x\leq 1, \; 0\leq y\leq x \right\}, \\
\mathcal{T}_{1} & = & \left\{(x,y): \; 1\leq x\leq 2, \; 0\leq y\leq 2-x \right\}.
\end{eqnarray*}

As a first approach to the study of the dynamics of $T_{a(s),b(s)}$ for $s$ close to $2$, certain family of piecewise linear maps was introduced in \cite{PuRoTaVi1}. This family was defined on the triangle $\mathcal{T}$ by
\begin{equation}\label{Lambda}
\Lambda_{t}(x,y)=
\left\{
\begin{array}{ll}
(t(x+y),t(x-y)) & \text{, if } (x,y)\in \mathcal{T}_{0}\\
(t(2-x+y),t(2-x-y)) & \text{, if } (x,y) \in \mathcal{T}_{1}.
\end{array}
\right.
\end{equation}
These maps can be seen as the composition of linear maps defined by the matrices
\begin{equation*} 
A_{t}=
\left(
\begin{array}{cc}
 t & t \\
 t & -t
\end{array}
\right)
\end{equation*}
with the \textit{fold} of the whole plane along the line $\mathcal{C}=\left\{(x,y)\in \mathbb{R}^{2}: \; x=1 \right\}$ given by
\begin{equation*} 
\mathcal{F}_{\mathcal{C}, \mathcal{O}}(x,y)=
\left\{
\begin{array}{ll}
(x,y) &  \text{, if } x<1,\\
(2-x,y) & \text{, if } x\geq 1.
\end{array}
\right.
\end{equation*}

The triangle $\mathcal{T}$ is invariant for these maps $\Lambda_{t}$ whenever $0\leq t\leq 1$. If $t\leq 1/\sqrt{2}$, then the dynamics of $\Lambda_{t}$ is simple. However, if $t>1/\sqrt{2}$, then the eigenvalues of the matrix $A_{t}$ have modulus greater than one and this fact gives rise to richer dynamics. In particular, as was proved in \cite{PuRoTaVi2} there appear strange attractors. In the case in which $t>1/\sqrt{2}$, the map $\Lambda_{t}$ is called, according to Section 4 in \cite{PuRoTaVi1}, an \textit{Expanding Baker Map} (\textit{EBM} for short). 

To begin with, let us recall the definition of \textit{fold}.

\begin{definition}\label{defifold}
Let $ \mathcal{K} \subset \mathbb{R}^2 $ be a compact and convex domain with nonempty interior, $ P $ a point in $ \mathcal{K} $ and $ \mathcal{L} $ a line with $ \mathcal{L} \cap int(\mathcal{K}) \neq \emptyset $ and $ P \notin \mathcal{L}.$ Then $ \mathcal{L} $ divides $ \mathcal{K} $ into two subsets denoted by $ \mathcal{K}_0 $ and $ \mathcal{K}_1 $ ($ \mathcal{K}_0 $ the one containing $ P $). We define the \textbf{fold} $ \mathcal{F}_{\mathcal{L},P} $ as the map
\begin{equation*}
\mathcal{F}_{\mathcal{L},P}(Q)=\left\{ 
\begin{array}{lcc} 
Q & , \ if \ Q \in \mathcal{K}_0 \\ 
\overline{Q} & , \ if \ Q \in \mathcal{K}_1
\end{array}
\right.
\end{equation*}
where $ \overline{Q} $ denotes the symmetric point of $ Q $ with respect to $ \mathcal{L} .$

In the above conditions, the map $ \mathcal{F}_{\mathcal{L},P} $ is said to be a \textbf{good fold} if $ \mathcal{F}_{\mathcal{L},P}(\mathcal{K}) = \mathcal{K}_0 .$
\end{definition}

Now, let us write $ \mathcal{L}=\mathcal{L}_1 $ and let  $ \mathcal{L}_2 $ be a line with $ \mathcal{L}_2 \cap int(\mathcal{K}_0) \neq \emptyset $ and $ P \notin \mathcal{L}_2. $ Then,  $ \mathcal{L}_2 $ divides $ \mathcal{K}_0 $ into two subsets $ \mathcal{K}_{00} $ and $ \mathcal{K}_{01} $ ($ \mathcal{K}_{00}$ denotes the one containing $ P $). Let us assume that  $ \mathcal{F}_{\mathcal{L}_2,P}(\mathcal{K}_0) = \mathcal{K}_{00} $
(i.e,  $ \mathcal{F}_{\mathcal{L}_2,P}$ is a good fold). Repeating these arguments, we may successively define a sequence of good folds $ \mathcal{F}_{\mathcal{L}_1,P} , \ldots , \mathcal{F}_{\mathcal{L}_n,P} $ where 
\begin{eqnarray*}
& & \mathcal{F}_{\mathcal{L}_1,P}: \mathcal{K} \rightarrow \mathcal{K}_0 \; ,\\[1.5ex]
& & \mathcal{F}_{\mathcal{L}_i,P}: \mathcal{K}_{0 \stackrel{i-1}{\ldots} 0} \rightarrow \mathcal{K}_{0 \stackrel{i}{\ldots} 0} \; , \  i = 2, \ldots, n
\end{eqnarray*}
\ \\[-2ex]
with $ \mathcal{K}_{0 \stackrel{i}{\ldots} 0} \subset \mathcal{K}_{0 \stackrel{i-1}{\ldots} 0} $ and $ P \in \mathcal{K}_{0 \stackrel{i}{\ldots} 0 } $ for every $ i=1,\ldots,n .$

We now recall the concept of \textit{EBM}.

\begin{definition}\label{defEBM}
Let $ \mathcal{K} \subset \mathbb{R}^2 $ be a compact and convex domain with nonempty interior. Let $ P $ be a point in $ \mathcal{K} $ and $ \{ \mathcal{F}_{\mathcal{L}_1,P} , \ldots , \mathcal{F}_{\mathcal{L}_n,P} \} $ a sequence of good folds of $ \mathcal{K} $ with $ P \in \mathcal{K}_{0 \stackrel{i}{\ldots} 0 } $ for every $ i=1 , \ldots , n .$ Let $ A: \mathbb{R}^2 \rightarrow \mathbb{R}^2 $ be an expanding linear map, i.e., $ |det(A)| > 1 .$ Let us consider 
\begin{equation*} 
\widetilde{A}: Q \in \mathbb{R}^2 \mapsto \widetilde{A}(Q)=P+A(Q-P)
\end{equation*}
and assume that $ \widetilde{A}(\mathcal{K}_{0 \stackrel{n}{\ldots} 0}) \subset \mathcal{K} .$ We define the \textbf{Expanding Baker Map} associated to $ P, \ A , \ \mathcal{L}_1,\ldots ,\mathcal{L}_n $ as the map $ \Gamma: \mathcal{K} \rightarrow \mathcal{K} $ given by
\begin{equation*}
\Gamma = \widetilde{A} \circ \mathcal{F}_{\mathcal{L}_n} \circ \ldots \circ \mathcal{F}_{\mathcal{L}_1} .
\end{equation*}
For short, we shall denote 
\begin{equation*}
\Gamma=EBM(\mathcal{K},\mathcal{L}_1,\ldots ,\mathcal{L}_n,P,A) .
\end{equation*}
\end{definition}

The choice of the family $\Lambda_{t}$, $0\leq t\leq 1$, was motivated in \cite{PuRoTaVi0}: the study of the dynamics exhibited by the family $\Lambda_{t}$ is mainly justified when one compares its attractors (numerically obtained in \cite{PuRoTaVi0}) with the attractors (numerically studied in \cite{PuTa2}) for the family $T_{a(s),b(s)}$ with $ s \in [0,2]$ and $(a(s),b(s))\in G,$ see (\ref{curve}). Both families of maps display convex strange attractors, connected (but non simply-connected) strange attractors and non-connected strange attractors (these last ones formed by numerous connected pieces).

A first analytical proof on the existence of a convex strange attractor of  $\Lambda_{t}$ was given in \cite{PuRoTaVi2} for every $t\in (t_{0},1]$, where $t_{0}=\frac{1}{\sqrt{2}}(1+\sqrt{2})^{1/4}$. The appearance of attractors with several pieces suggested the definition of renormalizable \textit{EBM} given in \cite{PuRoVi}. 

\begin{definition}\label{defiresdom}
Let $ \Gamma $ be a map defined in certain domain $ \mathcal{K} .$ We said that $ \mathcal{D} \subset \mathcal{K} $ is a \textbf{restrictive domain} if $ \mathcal{D} \neq \mathcal{K} $ and there exists $ k=k(\mathcal{D}) \in \mathbb{N} $ such that 
\begin{itemize}
\item[i)] $ \Gamma^j(\mathcal{D}) \cap \mathcal{D} = \emptyset $ for every $ j = 1 , \ldots , k-1 \; ,$
\item[ii)] $ \Gamma^k(\mathcal{D}) \subset \mathcal{D} .$
\end{itemize} 
\end{definition}

\begin{definition}\label{defiEBMren}
An \textit{EBM} $ \Gamma $ defined on certain domain $ \mathcal{K} $ is said to be \textbf{renormalizable} if there exists a restrictive domain $ \mathcal{D} $ (with an associated natural number $k=k(\mathcal{D}))$ such that $ \Gamma_{|\mathcal{D}}^k $ is, up to an affine change in coordinates, an \textit{EBM} defined on $ \mathcal{K} .$
\end{definition}

\begin{definition}\label{defiEBMnren}
Let $ \Gamma $ be a \textit{renormalizable EBM} with restrictive domain $ \mathcal{D} $ (with an associated natural number $k=k(\mathcal{D})).$  Let us denote $ \Gamma_1 = \Gamma_{|\mathcal{D}}^k .$ If $ \Gamma_1 $ is a \textit{renormalizable EBM}, we call $ \Gamma $ \textbf{twice renormalizable EBM}. Similarly we may speak about \textbf{three times renormalizable EBMs} , \ldots , \textbf{infinitely renormalizable EBMs}.
\end{definition}

In \cite{PuRoVi} (see the Main Theorem), it was proved that there exist three values of the parameter $t$, $\frac{1}{\sqrt{2}}<t_{3}<t_{2}<t_{1}=2^{-2/5}$, such that $\Lambda_{t}$ is a $n$ times renormalizable \textit{EBM} for every $t\in (\frac{1}{\sqrt{2}},t_n), n=1,2,3.$

The proof of this result is consequence of a renormalization procedure which allows us to understand how connected attractors may break up giving rise to new attractors formed by an increasing number of pieces. Furthermore, the proper renormalization method is fruitful to explain the coexistence of attractors: the renormalization can be simultaneously used on two disjoint restrictive domains to get two different attractors. Numerical evidences allowed us to conjecture that these attractors are strange. 
In fact, \cite{PuRoVi} was finished with the following three conjectures that we are going to prove in the present paper.

\begin{conjecture}\label{cona} 
There exists a decreasing sequence $\left\{ t_{n} \right\}_{n\in \mathbb{N}}$, convergent to $\frac{1}{\sqrt{2}}$ such that $\Lambda_{t} $ is a $n$ times renormalizable \textit{EBM} for every $t\in (\frac{1}{\sqrt{2}},t_{n})$.
\end{conjecture}

\begin{conjecture}\label{conb} 
There is no value of $t$ for which $\Lambda_{t} $ is infinitely many renormalizable.
\end{conjecture}

\begin{conjecture}\label{conc} 
For each natural number $n$ there exists an interval $I_{n}\subset (\frac{1}{\sqrt{2}},t_{n}$) such that $\Lambda_{t} $ displays, at least, $2^{n-1}$ different strange attractors.
\end{conjecture}

The proofs of these conjectures are again strongly based on the notion of renormalization. From \cite{PuRoVi} we know that the renormalization of $\Lambda_{t} $ leads to an \textit{EBM} with two folds. To be more precise, these two folds take place, respectively, along the lines
\begin{eqnarray*}
\mathcal{C} & = & \left\{ (x,y)\in \mathcal{T}: \; x=1 \right\} \\
\mathcal{L}(b) & = & \left\{ (x,y)\in \mathcal{T}_{0}: \; x+y=b \right\}.
\end{eqnarray*}
Therefore, we shall be mainly interested in the study of the two-parameter family of \textit{EBMs}
\begin{equation}\label{defF}
\mathbb{F}=\left\{\Psi_{a,b}=EBM(\mathcal{T},\mathcal{C},\mathcal{L}(b),\mathcal{O},B_{a}) \right\}
\end{equation}
where
\begin{equation}\label{defBa}
B_{a}=
\left(
\begin{array}{cc}
 a & 0 \\
 0 & a
\end{array}
\right),
\end{equation}
defines expansive linear maps with fixed point in $\mathcal{O}$. Let
$\mathcal{P}$ denote the set of parameters given by
\begin{equation}\label{defP}
\mathcal{P}=\left\{ (a,b)\in (1,2] \times [1,2]: \; ab\leq 2 \right\}.
\end{equation}
The choice of $\mathcal{P}$ guarantees that $\mathcal{T}$ is invariant by $\Psi_{a,b}$ for every $(a,b)\in \mathcal{P}$. See Section 2 for more details of the definition of $\Psi_{a,b}$.

To begin with, we shall prove the existence of two regions $\mathcal{P}_{j}$, $j=1,2$ of $\mathcal{P}$ in which $\Psi_{a,b}$ has an invariant rectangle $\mathcal{R}_{a,b}$ and, in addition, for each $(a,b)\in \mathcal{P}_{j}$,  $\Psi_{a,b}^2$ is conjugate on $\mathcal{R}_{a,b}$ to a direct product of tent maps. As a consequence, it follows that the maps $\Psi_{a,b}$ exhibit strange attractors on $\mathcal{R}_{a,b}$.

Later, we shall find a third region $\mathcal{P}_{3}$ contained in $\mathcal{P}$ such that for every $(a,b)\in \mathcal{P}_{3}$ the map $\Psi_{a,b}$ can be renormalized on two different restrictive domains $\Delta_{a,b}$ and $\Pi_{a,b}$ at the same time. In this way, we can define two different renormalization operators $H_{\Delta}$ and $H_{\Pi}$ and prove the following result.

\begin{mtema}\label{tmaa}
For every $(a,b)\in\mathcal{P}_{3}$ the map $\Psi_{a,b}$ is simultaneously renormalizable in $\mathbb{F}$ on two different restrictive domains. Namely:
\begin{itemize}
\item[i)] The restriction of $\Psi_{a,b}^{4}$ to $\Delta_{a,b}$ is conjugate by means of an affine change in coordinates to $\Psi_{H_{\Delta}(a,b)}$ restricted to $\Psi_{H_{\Delta}(a,b)}(\mathcal{T})$.
\item[ii)] The restriction of $\Psi_{a,b}^{4}$ to $\Pi_{a,b}$ is conjugate by means of an affine change in coordinates to $\Psi_{H_{\Pi}(a,b)}$.
\end{itemize}
\end{mtema}
   
The operators $H_{\Delta}$ and $H_{\Pi}$ satisfy fruitful properties used along the rest of the paper. In particular, $\mathcal{P}_{1}\subset H_{\Delta}(\mathcal{P}_{3})$ and $\mathcal{P}_{2}\subset H_{\Pi}(\mathcal{P}_{3})$. Moreover, both maps $H_{\Delta}$ and $H_{\Pi}$ has a repelling fixed point at $(1,\sqrt{2})$.

The maps $\Lambda_t$ defined in (\ref{Lambda}) satisfy, according to \cite{PuRoVi}, that $ \Lambda_{t}^{8} = \Psi_{a,b}$ for every $ t \in [\frac{1}{\sqrt{2}},\frac{1}{\sqrt[5]{4}}] $ with $ a=16t^{8} $ and $ b=1/2t^{3} .$ Notice that
\begin{equation}\label{defgam0}
\gamma_{0}=\left\{ \left(16t^{8},\frac{1}{2t^{3}} \right): \; t \in \left[\frac{1}{\sqrt{2}},\frac{1}{\sqrt[5]{4}} \right] \right\}
\end{equation}
is a curve contained in $\mathcal{P}$ starting at $(1,\sqrt{2})$. By using the relative positions between the tangent vector to $\gamma_{0}$ at $(1,\sqrt{2})$ and the eigenvectors of $DH_{\Delta}$ at this point, one concludes the existence of $n_{0}\in\mathbb{N}$ such that $ H_{\Delta}^{n}(\gamma_{0})\cap \mathcal{P}_{1}\neq \emptyset ,$ for every $n>n_{0}$. This is the main argument used in the proof of the following result.

\begin{mtema}\label{tmab} 
It holds that:
\begin{itemize}
\item[i)] There exists a decreasing sequence $\left\{ t_{n} \right\}_{n\in\mathbb{N}}$ converging to $\frac{1}{\sqrt{2}}$, such that $\Lambda_{t}$ is a $n$ times renormalizable \textit{EBM} for every $t\in (\frac{1}{\sqrt{2}},t_{n})$.
\item[ii)] For each natural number $n$ there exists an interval $ I_{n} $ such that $\Lambda_{t} $ displays, at least, $2^{n}$ different strange attractors whenever $t\in I_n.$
\end{itemize}
\end{mtema}

It is clear that the first statement of this theorem gives an affirmative answer to Conjecture \ref{cona}. Renormalizations are defined on bounded domains with nonempty interior. Infinitely many renormalizations of \textit{EBMs} can not be made on a domain with nonempty interior. Thus, if $ t > \frac{1}{\sqrt{2}} ,$ then $ \Lambda_t $ can not be an infinitely many renormalizable map. Hence, Conjecture \ref{conb} follows.  Finally, the second statement of Theorem \ref{tmab} proves a slightly weaker version of Conjecture \ref{conc}. In fact, it remains as an open problem to show that the interval of parameters $I_n$ in which the existence of at least $2^n$ strange attractors is proved can be constructed inside the set of parameters where, according to the first statement of Theorem \ref{tmab}, the map $\Lambda_t$ is a n times renormalizable map. Although we think this stronger result is also true, we remark that the most important dynamical property, i.e. the coexistence of any arbitrarily large number of persistent (in an interval of parameters) different strange attractors is demonstrated along this paper.

The paper is organized as follows. In Section \ref{seccoe} the subsets of parameters $\mathcal{P}_{1}$ and $\mathcal{P}_{2}$ are constructed and the coexistence of strange attractors for such parameters is proved. In Section \ref{secren} the proof of Theorem \ref{tmaa} is given and useful properties of the renormalization operators are stated. Finally, Section \ref{sectmab} is devoted to demonstrate Theorem \ref{tmab}. 

% ------------------------------ SECTION ------------------------------

\section{Regions of coexistence of strange attractors}\label{seccoe} \

In order to find regions of parameters for which the map $ \Psi_{a,b} $ displays coexistence of strange attractors we shall look for values of the parameter $ (a,b) $ such that $ \Psi_{a,b} $ may be conjugate to a direct product of one-dimensional tent maps.

% ------------------------------ SUBECTION ------------------------------

\subsection{Direct product of one-dimensional tent maps}

Let us recall the family of one-dimensional tent maps $ \{ \lambda_{\mu} \}_{\mu \in [0,2]} $ given by
\begin{equation*}
\lambda_{\mu} : x \in [0,2] \mapsto \lambda_{\mu}(x)=
\left\{
\begin{array}{ll}
\mu x & \mbox{, if } x \in [0,1] \\
\mu (2-x) & \mbox{, if } x \in [1,2] .
\end{array}
\right.
\end{equation*}
It is known that the interval $ I_{\mu}=[\mu(2-\mu),\mu] $ is an invariant set of $\lambda_{\mu} $ for every $ \mu \in (1,2] .$ Moreover, $ I_{\mu} $ is a strange attractor for every  $ \mu \in (\sqrt{2},2] $ and strange attractors with several pieces may be also obtained for values of the parameter $ \mu \in (1,\sqrt{2}] $ by using renormalization techniques. Furthermore, these strange attractors are strongly topologically mixing according to the next definition.

\begin{definition}
Let $ f: \mathcal{M} \rightarrow \mathcal{M} $ be a transformation, $\mathcal{M} \subset \mathbb{R}^N .$ Then, $ f $ is said to be \textbf{strongly topologically mixing} in an invariant set $ \mathcal{A} $ if for any open subset $ \mathcal{U} \subset \mathcal{A} $ there is some natural number $ k $ such that $ f^k(\mathcal{U}) = \mathcal{A} .$
\end{definition}

From now on, we shall denote by $ \Gamma_{\mu} $ the two dimensional map defined by 
\begin{equation}\label{defgamt}
\Gamma_{\mu} : (x,y) \in [0,2] \times [0,2] \mapsto \Gamma_{\mu} (x,y) = ( \lambda_{\mu} (x) , \lambda_{\mu} (y) ) .
\end{equation}

\begin{lemma}\label{lemgammat}
The following statements hold:
\begin{itemize}
\item[a)] For every $ \mu \in (\sqrt{2},2] $ the map $ \Gamma_{\mu} $ has a strongly topologically mixing strange attractor with two positive Lyapounov exponents. 
\item[b)] For every natural number $ n $ and for every $ \mu \in (2^{\frac{1}{2^{n+1}}},2^{\frac{1}{2^n}}] $ the map $ \Gamma_{\mu} $ is $ n $ times renormalizable and displays $ 2^n $ strange attractors with two positive Lyapounov exponents.
\end{itemize}
\end{lemma}

\begin{proof}
It is clear that 
\begin{equation}\label{defst}
\mathcal{S}_{\mu} = I_{\mu} \times I_{\mu} = [\mu(2-\mu),\mu] \times [\mu(2-\mu),\mu]
\end{equation}
is invariant by $ \Gamma_{\mu} .$ Moreover, since $ \lambda_{\mu} $ is strongly topologically mixing on $ I_{\mu} $ it is easy to deduce that $ \Gamma_{\mu} $ is strongly topologically mixing on $ \mathcal{S}_{\mu} .$ 

According to \cite{Buz2} the map $ \Gamma_{\mu} $ has a unique absolutely continuous and ergodic measure $ \widetilde{\nu} $ with support equal to $ \mathcal{S}_{\mu} .$ This measure $ \widetilde{\nu} $ coincides with $ \nu \times \nu ,$ where $ \nu $ is the unique absolutely continuous and ergodic measure of $ \lambda_{\mu} .$ The rest of the proof of the first statement follows in the same way as Proposition 3 in \cite{PuTa1}.

Now, if $ \mu \in (2^{\frac{1}{2^{n+1}}},2^{\frac{1}{2^n}}] $ it is known that there exists a restrictive domain $ J_{\mu} $ such that $ \lambda_{\mu}^{2^n} $ restricted to $ J_{\mu} $ is conjugate to $ \lambda_{\mu^{2^n}} $ (see \cite{BrBr}, \cite{MeSt} or \cite{Mel} for related details). Therefore, $ \lambda_{\mu}^{2^n} $ displays a strongly topologically mixing strange attractor $ J_{\mu,0} \subset J_{\mu} .$ Hence, $ \mathcal{S}_{\mu,0} = J_{\mu,0} \times J_{\mu,0} $ is a strongly topologically mixing strange attractor for $ \Gamma_{\mu}^{2^n} $ and, defining $ J_{\mu,k}=\lambda_{\mu}^k(J_{\mu,0}) $ for $ k = 0 \ldots 2^n-1 ,$ the same holds for every $ \mathcal{S}_{\mu,k} = J_{\mu,0} \times J_{\mu,k}. $ Then, it is easy to see that
\begin{equation*}
\widetilde{\mathcal{S}}_{\mu,k} = \bigcup_{j=1}^{2^n} \Gamma_{\mu}^j(\mathcal{S}_{\mu,k}), 
\end{equation*}
$ k = 0 \ldots 2^n-1 ,$ are $ 2^n $ different strange attractors for $ \Gamma_{\mu} .$
\end{proof}

\begin{remark}\label{rematm2d}
Let us note that for every $ \mu \in (1,2] $ the map $ \Gamma_{\mu} $ can be written as the \textit{EBM} with two folds $ \Gamma_{\mu}=EBM(\mathcal{Q},\mathcal{L}_1,\mathcal{L}_2,\mathcal{O},B_{\mu}) ,$ see also \eqref{defF}, where 
\begin{eqnarray*}
\mathcal{Q} & = & [0,2] \times [0,2] \; , \\[1.5ex]
\mathcal{L}_1 & = & \{(x,y) \in \mathcal{Q} : \ x=1 \} \; , \\[1.5ex]
\mathcal{L}_2 & = & \{(x,y) \in \mathcal{Q} : \ x \leq 1 \; , \; y=1 \} \; , \\[1.5ex]
B_{\mu} & = & \left( 
\begin{array}{cc} 
\mu & 0 \\
0 & \mu
\end{array}
\right) \; .
\end{eqnarray*}
\end{remark}
% --------------- FIGURE ------------------------
\begin{figure}[!ht]
\begin{minipage}{1 \linewidth}
\begin{center}
{\small
\begin{tikzpicture}[xscale=1,yscale=1]

\fill [mycolor] (0.0,0.0) -- (2.0,0.0) -- (2.0,2.0) -- (0.0,2.0) -- (0.0,0.0); 
\draw [darkgray,thick] (0.0,0.0) -- (2.0,0.0) -- (2.0,2.0) -- (0.0,2.0) -- (0.0,0.0); 
\filldraw [black] (0.0,0.0) circle (1pt);
\node [black] at (0.0,-0.3) {$ \mathcal{O} $}; 

\draw [darkgray,thick,dashed] (1.0,0.0) -- (1.0,2.0);
\draw [darkgray,thick,dashed] (0.0,1.0) -- (1.0,1.0); 

\node [black] at (1.3,1.5) { $ \mathcal{L}_1 $ };
\node [black] at (0.5,0.7) { $ \mathcal{L}_2 $ };

\node [black] at (2.5,1.3) { $ \mathcal{F}_{\mathcal{L}_1} $ };
\draw[black, -latex', line width=3pt] (2.2,1.0) -- (2.8,1.0);

\draw [darkgray,thick] (3.0,0.0) -- (5.0,0.0) -- (5.0,2.0) -- (3.0,2.0) -- (3.0,0.0); 
\fill [mycolor] (3.0,0.0) -- (4.0,0.0) -- (4.0,2.0) -- (3.0,2.0) -- (3.0,0.0); 
\draw [darkgray,thick] (3.0,0.0) -- (4.0,0.0) -- (4.0,2.0) -- (3.0,2.0) -- (3.0,0.0); 
\filldraw [black] (3.0,0.0) circle (1pt);
\node [black] at (3.0,-0.3) {$ \mathcal{O} $}; 

\draw [darkgray,thick,dashed] (3.0,1.0) -- (4.0,1.0); 

%\node [black] at (4.3,1.5) { $ \mathcal{L}_1 $ };
\node [black] at (3.5,0.7) { $ \mathcal{L}_2 $ };

\node [black] at (5.5,1.3) { $ \mathcal{F}_{\mathcal{L}_2} $ };
\draw[black, -latex', line width=3pt] (5.2,1.0) -- (5.8,1.0);

\draw [darkgray,thick] (6.0,0.0) -- (8.0,0.0) -- (8.0,2.0) -- (6.0,2.0) -- (6.0,0.0); 
\fill [mycolor] (6.0,0.0) -- (7.0,0.0) -- (7.0,1.0) -- (6.0,1.0) -- (6.0,0.0); 
\draw [darkgray,thick] (6.0,0.0) -- (7.0,0.0) -- (7.0,1.0) -- (6.0,1.0) -- (6.0,0.0); 
\filldraw [black] (6.0,0.0) circle (1pt);
\node [black] at (6.0,-0.3) {$ \mathcal{O} $}; 

\node [black] at (8.5,1.3) { $ B_{\mu} $ };
\draw[black, -latex', line width=3pt] (8.2,1.0) -- (8.8,1.0);

\draw [darkgray,thick] (9.0,0.0) -- (11.0,0.0) -- (11.0,2.0) -- (9.0,2.0) -- (9.0,0.0); 
\fill [mycolor] (9.0,0.0) -- (10.5,0.0) -- (10.5,1.5) -- (9.0,1.5) -- (9.0,0.0); 
\draw [darkgray,thick] (9.0,0.0) -- (10.5,0.0) -- (10.5,1.5) -- (9.0,1.5) -- (9.0,0.0); 
\filldraw [black] (9.0,0.0) circle (1pt);
\node [black] at (9.0,-0.3) {$ \mathcal{O} $};

\filldraw [black] (10.5,1.5) circle (1pt);
\node [black] at (10.5,1.75) {$ (\mu,\mu) $};

\end{tikzpicture} }
\end{center}
\caption{Dynamics of $ \Gamma_{\mu} .$}
\label{figdyngammat}
\end{minipage}
\end{figure}
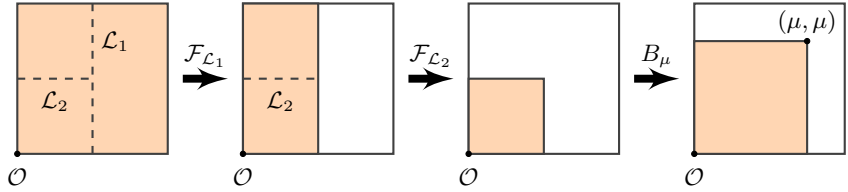

% ------------------------------ SUBECTION ------------------------------

\subsection{The two-parameter family of \textit{EBMs}}

From now on, $ \Omega(Q^1,Q^2,\dots,Q^n)$ denotes the polygonal set in $ \mathbb{R}^2 $ with consecutive vertices $Q^1$, $Q^2$,\dots, $Q^n.$ Let us recall from \cite{PuRoVi} the two-parameter family of \textit{EBMs}
\begin{equation*}
\mathbb{F} = \{ \Psi_{a,b} = EBM(\mathcal{T},\mathcal{C},\mathcal{L}(b),\mathcal{O},B_a)\} : (a,b) \in \mathcal{P} \}
\end{equation*}
where 
\begin{eqnarray*}
\mathcal{T} & = & \Omega(\mathcal{O},(1,1),(2,0)) \; , \\[1.5ex]
\mathcal{C} & = & \{(x,y) \in \mathcal{T} : \; x=1 \} \; , \\[1.5ex]
\mathcal{L}(b) & = & \{(x,y) \in \mathcal{T} : \;  x \leq 1 \mbox{ and } x+y=b \} \; , \\[1.5ex]
\mathcal{O} & = & (0,0) \; , \\[1.5ex]
B_a & = & \left( 
\begin{array}{cc} 
a & 0 \\
0 & a
\end{array}
\right) \; ,
\end{eqnarray*}
and $ \mathcal{P} $ (see \eqref{defP}) denotes the set of parameters given by
\begin{equation*}
\mathcal{P} = \{ (a,b) \in (1,2] \times [1,2] : \ ab \leq 2  \}.
\end{equation*} 
The choice of $ \mathcal{P} $ guarantees that $ \mathcal{T} $ is invariant by $ \Psi_{a,b} $ for every $ (a,b) \in \mathcal{P} .$ 

We shall also consider
\begin{eqnarray*}
\mathcal{T}_{0} & = & \left\{ (x,y) \in \mathcal{T} : \; 0 \leq x \leq 1 \right\} \; , \\[1.5ex] 
\mathcal{T}_{1} & = & \left\{ (x,y) \in \mathcal{T} : \; 1 \leq x \leq 2 \right\} \; , \\[1.5ex]
\mathcal{L}^{\prime}(b) & = & \{(x,y) \in \mathcal{T}_1 : \ x-y=2-b \}.
\end{eqnarray*}
Then, we can split $ \mathcal{T} $ in four domains, see Figure \ref{figpab},
\begin{eqnarray*}
\mathcal{T}_0^- & = & \{ (x,y) \in \mathcal{T}_0 : \ x+y \leq b \} \; , \\
\mathcal{T}_0^+ & = & \{ (x,y) \in \mathcal{T}_0 : \ x+y \geq b \} \; , \\
\mathcal{T}_1^- & = & \{ (x,y) \in \mathcal{T}_1 : \ x-y \geq 2-b \} \; , \\
\mathcal{T}_1^+ & = & \{ (x,y) \in \mathcal{T}_1 : \ x-y \leq 2-b \} 
\end{eqnarray*}
in such a way that each map $ \Psi_{a,b} $ in $ \mathbb{F} $ is defined by 
\begin{equation}\label{defpsiab}
\Psi_{a,b}(x,y) = \left\{
\begin{array}{ll}
\left( ax , ay \right) 				& \mbox{, if } (x,y) \in \mathcal{T}_0^- \\
\left( a(b-y) , a(b-x) \right) 		& \mbox{, if } (x,y) \in \mathcal{T}_0^+ \\
\left( a(2-x) , ay \right) 			& \mbox{, if } (x,y) \in \mathcal{T}_1^- \\
\left( a(b-y) , a(b-2+x) \right) 	& \mbox{, if } (x,y) \in \mathcal{T}_1^+ \; .
\end{array}
\right.
\end{equation}

% --------------- FIGURE ------------------------
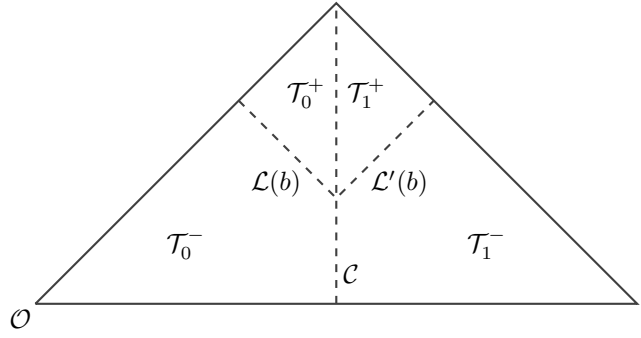
\begin{figure}[!ht]
\begin{minipage}{1 \linewidth}
\centering
\begin{tikzpicture}[xscale=4,yscale=4]
                %axis
\draw [darkgray,thick] (0,0) -- (2,0) -- (1,1) -- (0,0);
\draw [darkgray,thick,dashed] (1,0) -- (1,1);
\draw [darkgray,thick,dashed] (0.6757,0.6757) -- (1,0.3514) -- (1.3243,0.6757);

%\filldraw [black] (0.6757,0.6757) circle (0.25pt);
%\node[black] at (0.6757-0.12,0.6757+0.1) {$(\frac{1}{2t},\frac{1}{2t})$};
%\node[black] at (1.7,0.72) {$y=x-\frac{2t-1}{t}$};
%\draw [->,darkgray,thick] (1.6,0.65) .. controls (1.5,0.55) .. (1.3,0.62);
\node[black] at (-0.05,-0.05) {$\mathcal{O} $};
\node[black] at (0.8,0.40) {$\mathcal{L}(b) $};
\node[black] at (1.21,0.40) {$\mathcal{L}^{\prime}(b) $};
\node[black] at (1.05,0.1) {$\mathcal{C}$};
\node[black] at (0.5,0.2) {$\mathcal{T}_0^-$};
\node[black] at (0.9,0.7) {$\mathcal{T}_0^+$};
\node[black] at (1.5,0.2) {$\mathcal{T}_1^-$};
\node[black] at (1.1,0.7) {$\mathcal{T}_1^+$};
\end{tikzpicture}
\caption{The smoothness domains for a map in $ \mathbb{F} .$}
\label{figpab}
\end{minipage}
\end{figure}
% --------------- FIGURE ------------------------

As we have seen in \cite{PuRoVi}, Lemma 4.4, for every $ \Psi_{a,b} $ in $ \mathbb{F} $ there exists a unique fixed point $ P_{a,b} $ in $ int(\mathcal{T}) $ given by 
\begin{equation}\label{defPab}
P_{a,b} = (x_{a,b},y_{a,b}) = \left( \frac{ab+2a^2-a^2b}{1+a^2} , \frac{ab-2a+a^2b}{1+a^2} \right).
\end{equation}

Let us now define the points
\begin{eqnarray*}
C & = & (2-a-b+ab,a(b-1)) \; , \\ 
D & = & (1,a(b-1)) \; , \\
E & = & (1,ab-1) \; , \\ 
F & = & \left(\frac{1}{2}(2-b+ab),\frac{1}{2}(-2+b+ab)\right).
\end{eqnarray*}
As it has been proved in \cite{PuRoVi}, Section 4, for every point $ (x,y) $ in the interior of $ \mathcal{T} $ there exists a natural number $ n $ such that $ \Psi_{a,b}^n(x,y) $ belongs to certain invariant domain $ \mathcal{R}_1 $ given by, see Figure \ref{figR1},
\begin{equation}\label{defR1}
\mathcal{R}_1=\Omega(C_1,D_1,E_1,F_1)
\end{equation}
where
\begin{eqnarray*}
C_1 & = & \Psi_{a,b}(C) = (a(a+b-ab),a^{2}(b-1)) \; , \\ 
D_1 & = & \Psi_{a,b}(D) = (a(a+b-ab),a(b-1)) \; , \\
E_1 & = & \Psi_{a,b}(E) = (a(1+b-ab),a(b-1)) \; , \\ 
F_1 & = & \Psi_{a,b}(F) = \left(\frac{a}{2}(2+b-ab),\frac{a}{2}(-2+b+ab)\right).
\end{eqnarray*}

% --------------- FIGURE figR1 ------------------------

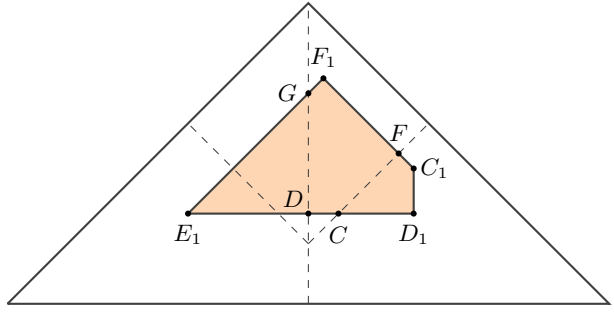
\begin{figure}[!ht]
\centering
\begin{minipage}{0.9 \linewidth}
\centering
{\small
\begin{tikzpicture}[xscale=4,yscale=4]

\fill[mycolor] (1.0500,0.7500) -- (1.3500,0.4500) -- (1.3500,0.3000) -- (0.6000,0.3000) ;
\draw[darkgray,thick] (1.0500,0.7500) -- (1.3500,0.4500) -- (1.3500,0.3000) -- (0.6000,0.3000) -- (1.0500,0.7500);

\draw[darkgray,thick] (0.0000,0.0000) -- (2.0000,0.0000) -- (1.0000,1.0000) -- (0.0000,0.0000); 
\draw[darkgray,dashed] (1.0000,0.0000) -- (1.0000,1.0000); 
\draw[darkgray,dashed] (1.0000,0.2000) -- (0.6000,0.6000); 
\draw[darkgray,dashed] (1.0000,0.2000) -- (1.4000,0.6000); 

\fill[black] (1.1000,0.3000) circle (0.10mm); 
\node[black] at  (1.1000,0.3000-0.07) {$ C $}; 
\fill[black] (1.3500,0.4500) circle (0.10mm); 
\node[black] at  (1.3500+0.07,0.4500) {$ C_1 $}; 
\fill[black] (1.0000,0.3000) circle (0.10mm); 
\node[black] at  (1.0000-0.05,0.3000+0.05) {$ D $}; 
\fill[black] (1.3500,0.3000) circle (0.10mm); 
\node[black] at  (1.3500,0.3000-0.07) {$ D_1 $}; 
\fill[black] (0.6000,0.3000) circle (0.10mm); 
\node[black] at  (0.6000,0.3000-0.07) {$ E_1 $}; 
\fill[black] (1.3000,0.5000) circle (0.10mm); 
\node[black] at  (1.3000,0.5000+0.07) {$ F $}; 
\fill[black] (1.0500,0.7500) circle (0.10mm); 
\node[black] at  (1.0500,0.7500+0.07) {$ F_1 $}; 
\fill[black] (1.0000,0.7000) circle (0.10mm); 
\node[black] at  (1.0000-0.07,0.7000) {$ G $}; 
\end{tikzpicture}
}
\caption{The set $ \mathcal{R}_1 .$}
\label{figR1}
\end{minipage}
\end{figure}

% ------------------------------ SUBECTION ------------------------------

\subsection{The region $ \mathcal{P}_1 $}

From now on, given two different points $ Q^1 $ and $ Q^2 $ in $\mathbb{R}^2$ we shall denote by $\overline{Q^1Q^2}$ the straight segment joining $Q^1$ and $Q^2.$ 

Given $ (a,b) \in \mathcal{P} $ we may compute the image of $ \mathcal{R}_1 $ by $ \Psi_{a,b} .$ It is easy to check that $ F_1 $ belongs to $ \mathcal{T}_1^+ $ for every $ (a,b) \in \mathcal{P} .$ Then, 
\begin{equation*}
F_2=\Psi_{a,b}(F_1)=\left(\frac{a}{2}(2a+2b-ab-a^{2}b),\frac{a}{2}(-4+2a+2b+ab-a^{2}b) \right).
\end{equation*}

Let us define the family of maps $ \mathbb{F}_1 = \{ \Psi_{a,b} \in \mathbb{F} : \; (a,b) \in \mathcal{P}_1 \} $ where, see Figure \ref{figsetP1},
\begin{equation}\label{defFdelta}
\mathcal{P}_1 = \left\{ (a,b) \in \mathcal{P} : \; b \leq \frac{2+2a}{2a+a^2} \right\}.
\end{equation}
Then, $ (a,b) \in \mathcal{P}_1 $ if and only if $ F_2 $ belongs to $ \mathcal{T}_{1}^+ .$ In this case, it is easy to check that
\begin{equation}\label{defR2del}
\mathcal{R}_2=\Psi_{a,b}(\mathcal{R}_1)=\Omega(C_1,D_1,G_1,F_2,F_1) \subset \mathcal{T}_{1}
\end{equation}
is an invariant set, where $ G_1 = (a(-1 + 2a + b - a^2b),a(b-1)).$ The aim of this section is to prove that for every $ (a,b) \in \mathcal{P}_1 $ there exists certain domain where $ \Psi_{a,b}^2 $ may be conjugate to a direct product of two one-dimensional tent maps. % To this end, it is enough to check that $ D_2 = \Psi_{a,b}(D_1) \in \mathcal{R}_2 .$ 

\begin{figure}[!ht]
\centering
\begin{minipage}{1 \linewidth}
\centering
\includegraphics[width=0.75\textwidth]{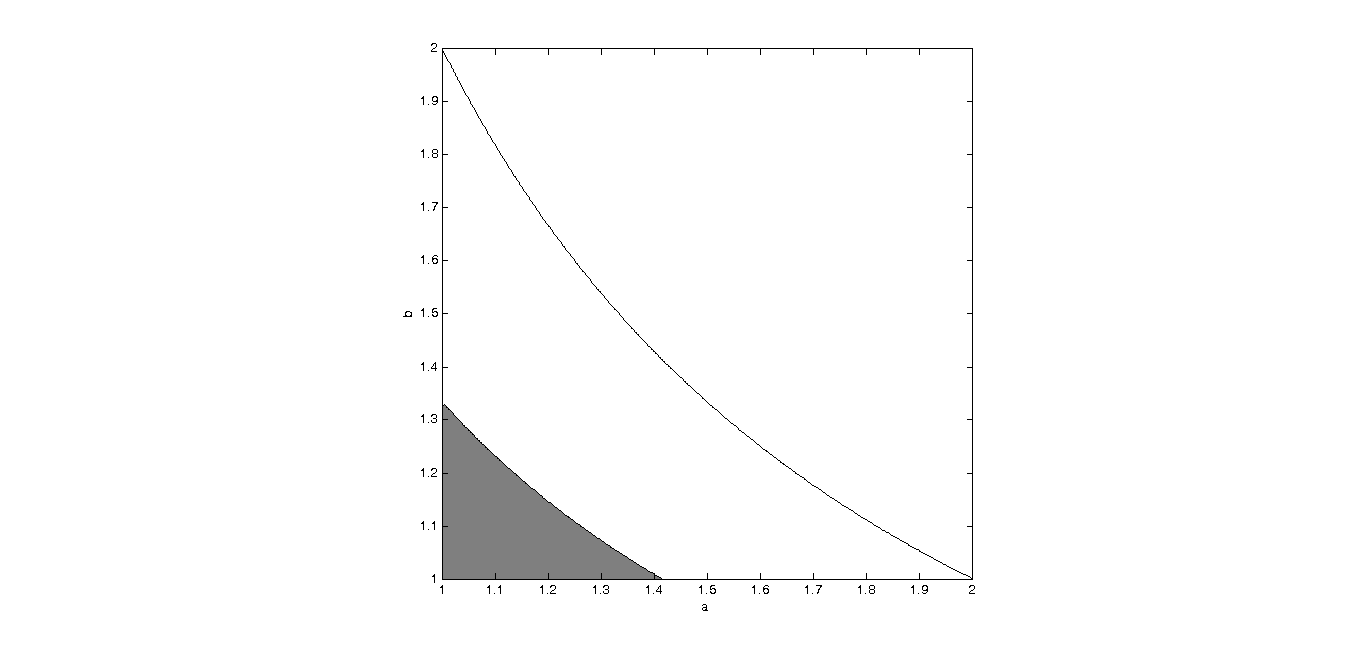}
\caption{The set of parameters $ \mathcal{P}_1 .$}
\label{figsetP1}
\end{minipage}
\end{figure}

Let us compute 
\begin{eqnarray}
F_3 = \Psi_{a,b}(F_2) & = & \left(\frac{a}{2}(4a-2a^2+2b-2ab-a^2b+a^3b) , \right. \hspace{3cm} \nonumber \\
& & \hspace{1cm} \left. \frac{a}{2}(-4+2a^2+2b+2ab-a^2b-a^3b) \right). \label{defF3}
\end{eqnarray}
It is easy to check that $ F_3 $ belongs to $ \mathcal{T}_1^- $ and we define the point 
\begin{equation*}
K = \left(\frac{1}{2}(2-2a+2a^2-b+2ab-a^3b),\frac{1}{2}(-2-2a+b+2a^2+2ab-a^3b) \right)
\end{equation*} 
as the intersection between $ \overline{F_2F_3} $ and $ \mathcal{L}^{\prime}(b) .$ In this situation, we consider the set, see Figure \ref{figRdelta},
\begin{equation}\label{defRdelta}
\mathcal{R}_{a,b} = \Omega(F_1,F_2,F_3,K_1)
\end{equation}
being \begin{eqnarray*}
K_1 & = & \Psi_{a,b}(K) = \\
& = & \left(\frac{a}{2}(2+2a-2a^2+b-2ab+a^3b),\frac{a}{2}(-2-2a+b+2a^2+2ab-a^3b) \right).
\end{eqnarray*}

% --------------- FIGURE figZ2 ------------------------
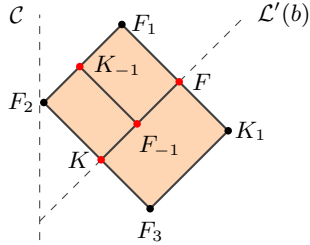
\begin{figure}[!ht]
\centering
\begin{minipage}{0.9 \linewidth}
\centering
{\small
\begin{tikzpicture}[xscale=10,yscale=10]

\fill[mycolor,thick] (1.1093,0.2868) -- (1.0055,0.1830) -- (1.1468,0.0417) -- (1.2506,0.1454); 
\draw[darkgray,thick] (1.1093,0.2868) -- (1.0055,0.1830) -- (1.1468,0.0417) -- (1.2506,0.1454) -- (1.1093,0.2868); 

\draw [darkgray,dashed] (1.0000,0.0000) -- (1.0000,0.3000);
\draw [darkgray,dashed] (1.0000,0.0251) -- (1.0817,0.1068);
\draw [darkgray,dashed] (1.1855,0.2106) -- (1.2749,0.3000);

\draw[darkgray,thick] (1.0817,0.1068) -- (1.1855,0.2106); 
\draw[darkgray,thick] (1.1295,0.1546) -- (1.0533,0.2308); 

\fill[red] (1.1295,0.1546) circle (0.05mm); 
\node[black] at  (1.1295+0.03,0.1546-0.03) {$ F_{-1} $}; 
\fill[red] (1.1855,0.2106) circle (0.05mm); 
\node[black] at  (1.1855+0.03,0.2106) {$ F $}; 
\fill[black] (1.1093,0.2868) circle (0.05mm); 
\node[black] at  (1.1093+0.03,0.2868) {$ F_1 $}; 
\fill[black] (1.0055,0.1830) circle (0.05mm); 
\node[black] at  (1.0055-0.03,0.1830) {$ F_2 $}; 
\fill[black] (1.1468,0.0417) circle (0.05mm); 
\node[black] at  (1.1468,0.0417-0.03) {$ F_3 $}; 
\fill[red] (1.0533,0.2308) circle (0.05mm); 
\node[black] at  (1.0533+0.05,0.2308) {$ K_{-1} $}; 
\fill[red] (1.0817,0.1068) circle (0.05mm); 
\node[black] at  (1.0817-0.03,0.1068) {$ K $}; 
\fill[black] (1.2506,0.1454) circle (0.05mm); 
\node[black] at  (1.2506+0.03,0.1454) {$ K_1 $};

\node[black] at  (1.0000-0.03,0.3000) {$ \mathcal{C} $};
\node[black] at  (1.2749+0.05,0.3000) {$ \mathcal{L}^{\prime}(b) $};
\end{tikzpicture}
}
\caption{The set $ \mathcal{R}_{a,b} .$}
\label{figRdelta}
\end{minipage}
\end{figure}
% --------------- END FIGURE figZ2 ---------------------

\begin{lemma}\label{leminvdelta}
For every $ (a,b) \in \mathcal{P}_1 $ the set $ \mathcal{R}_{a,b} $ defined in \eqref{defRdelta} is strictly invariant by $ \Psi_{a,b} .$ Moreover, $ \mathcal{R}_{a,b} $ captures the orbit of any point in the interior of $ \mathcal{T}:$ for every point $ (x,y) $ in the interior of $ \mathcal{T} $ there exists a natural number $ n $ such that $ \Psi_{a,b}^n(x,y) $ belongs to $ \mathcal{R}_{a,b}. $
\end{lemma}

\begin{proof}
To prove that $ \mathcal{R}_{a,b} $ is strictly invariant we only need to check that $ F_3 $ belongs to the segment $ \overline{KF_2^{\prime}} ,$ where $ F_2^{\prime} $ is the symmetric point of $ F_2 $  with respect to $ \mathcal{L}^{\prime}(b) .$ 

Since $ K $ is the middle point of the segment $ \overline{F_2F_2^{\prime}},$ it is easy to check that the abscise of $ 
F_2^{\prime} $ is 
\begin{equation*} 
x_2^{\prime}=\frac{1}{2}(4-4a+2a^2+2ab-2b+a^2b-a^3b).
\end{equation*}

Finally, by comparing $ x_2^{\prime} $ to the abscise of $ F_3 $ (see \eqref{defF3}) one may check that $ F_3 $ belongs to the segment $ \overline{KF_2^{\prime}} $ whenever
\begin{equation*} 
2 \leq (1+a)b .
\end{equation*}
Notice that this inequality holds for every $ (a,b) \in \mathcal{P}_1 .$

Now, let us note that from Section 4.1 in \cite{PuRoVi}, $ \mathcal{R}_1 $ captures the orbit of every point in $ int(\mathcal{T}) $ and therefore the same holds for $ \mathcal{R}_2 $ (see \eqref{defR2del}). On the other hand, $ \mathcal{R}_{a,b} \cap \mathcal{T}_1^+ = \mathcal{R}_2 \cap \mathcal{T}_1^+ = \Omega(F,F_1,F_2,K). $ In order to prove that $ \mathcal{R}_{a,b} $ captures the orbit of any point of $ \mathcal{T} ,$ it is enough to check that for every point $ (x,y) \in \mathcal{R}_2 \cap \mathcal{T}_1^- $ there exists a natural number such that $ \Psi_{a,b}^n (x,y) \in \Omega(F,F_1,F_2,K). $ This last claim immediately follows from the fact that, see \eqref{defpsiab}, there is no orbit contained in $ \mathcal{T}_1^- $ and also that $ \mathcal{R}_2 $ is invariant.
\end{proof}

\begin{lemma}\label{lemconjdelta}
For every $ (a,b) \in \mathcal{P}_1$ there exists an affine change in coordinates such that $ \Psi_{a,b}^2 $ restricted to $ \mathcal{R}_{a,b} $ transforms into the map $ \Gamma_{a^2} $ restricted to $ \mathcal{S}_{a^2} $ (see \eqref{defgamt} and \eqref{defst}).
\end{lemma}

\begin{proof}
Let us first note that $ \Psi_{a,b} $ displays a fixed point in the boundary of $ \mathcal{T}_1^- $ given by
\begin{equation*}\label{def2perdelta}
P_2=(x_2,0)=\left(\frac{2a}{1+a},0\right).
\end{equation*}

On the other hand, the preimage of $ \mathcal{L}^{\prime}(b) $ in $ \mathcal{T}_1^+ $ is given by 
\begin{equation*}\label{defpreldelta}
\mathcal{L}^{-1} = \{ (x,y) \in \mathcal{T}_1^+ : \ x+y=d \}
\end{equation*}
where $ d=a^{-1}(2a+b-2) .$ Moreover, $ \mathcal{R}_{a,b} \cap \mathcal{L}^{-1} = \overline{K_{-1}F_{-1}} $ being $ K_{-1} \in \mathcal{T}_1^{+} $ and $ F_{-1} \in \mathcal{L}^{\prime}(b) $ such that $ \Psi_{a,b}(K_{-1})=K $ and $ \Psi_{a,b}(F_{-1})=F .$ Finally, $ \mathcal{R}_{a,b} \cap \mathcal{L}^{\prime}(b) = \overline{KF} .$

Let us consider the change in coordinates 
\begin{equation}\label{defome1}
(X,Y)=\omega_{a,b}(x,y)=\left( \frac{x_2-(x-y)}{x_2-(2-b)},\frac{x+y-x_2}{d-x_2} \right).
\end{equation}

It is a simple calculation to check that 
\begin{eqnarray*}
\omega_{a,b}(F_1) & = & (a^2,a^2) \; , \\ 
\omega_{a,b}(F_2) & = & (a^2,a^2(2-a^2)) \; , \\
\omega_{a,b}(F_3) & = & (a^2(2-a^2),a^2(2-a^2)) \; , \\ 
\omega_{a,b}(K_1) & = & (a^2(2-a^2),a^2) \; .
\end{eqnarray*}
Therefore, $ \omega_{a,b}(\mathcal{R}_{a,b})= \mathcal{S}_{a^2} $ (see \eqref{defst}). Moreover,
\begin{eqnarray*}
\omega_{a,b}(\overline{KF}) & = & \{ (X,Y) \in \mathcal{S}_{a^2} : \ X=1 \} \; , \\
\omega_{a,b}(\overline{K_{-1}F_{-1}}) & = & \{ (X,Y) \in \mathcal{S}_{a^2} : \ X \geq 1 , \; Y=1 \} .
\end{eqnarray*}

Let us consider a point $ (x,y) \in \Omega(F,F_1,K_{-1},F_{-1}).$ On one hand, 
\begin{equation*}
\Psi_{a,b}(x,y)=(a(b-y),a(b-2+x)) \in \Omega(F_1,F_2,K,F) \subset \mathcal{T}_1^+
\end{equation*}
and therefore 
\begin{equation*}
\Psi_{a,b}^2(x,y)=(a(b-a(b-2+x)),a(b-2+a(b-y)).
\end{equation*}
On the other hand, $ (X,Y)=\omega_{a,b}(x,y) $ satisfies $ X \geq 1 $ and $ Y \geq 1 $ and, consequently,
\begin{equation*}
\Gamma_{a^2}(X,Y)=(a^2(2-X),a^2(2-Y)).
\end{equation*}
Then, one may check that $ \omega_{a,b} \circ \Psi_{a,b}^2(x,y) = \Gamma_{a^2}(X,Y) .$ % \Gamma_{a^2} \circ \omega_{a,b}(x,y)

Now, let us consider a point $ (x,y) \in \Omega(F_2,K,F_{-1},K_{-1}) $ and let 
\begin{equation*}
(\widetilde{x},\widetilde{y})=(d-y,d-x)
\end{equation*}
be the symmetric point of $ (x,y) $ with respect to $ \mathcal{L}^{-1} .$ Then, $ \Psi_{a,b}^2(x,y)= \Psi_{a,b}^2(\widetilde{x},\widetilde{y}). $ Moreover, $ (X,Y)=\omega_{a,b}(x,y) $ and $ (\widetilde{X},\widetilde{Y})=\omega_{a,b}(\widetilde{x},\widetilde{y}) $ are symmetric points in $ \mathcal{S}_{a^2} $ with respect to $ Y = 1 $ and therefore
\begin{equation*}
\omega_{a,b} \circ \Psi_{a,b}^2(x,y)= \omega_{a,b} \circ \Psi_{a,b}^2(\widetilde{x},\widetilde{y})=\Gamma_{a^2}(\widetilde{X},\widetilde{Y}) = \Gamma_{a^2}(X,Y) .
\end{equation*}

Finally, let us consider a point $ (x,y) \in \Omega(F,K,F_3,K_1) $ and let
\begin{equation*}
(\widetilde{x},\widetilde{y})=(y+d,x-d)
\end{equation*}
be the symmetric point of $ (x,y) $ with respect to $ \mathcal{L}^{\prime}(b) .$ Then, $ \Psi_{a,b}^2(x,y)= \Psi_{a,b}^2(\widetilde{x},\widetilde{y}). $ Furthermore, $ (X,Y)=\omega_{a,b}(x,y) $ and $ (\widetilde{X},\widetilde{Y})=\omega_{a,b}(\widetilde{x},\widetilde{y}) $ are symmetric points in $ \mathcal{S}_{a^2} $ with respect to $ X = 1 $ and therefore
\begin{equation*}
\omega_{a,b} \circ \Psi_{a,b}^2(x,y)= \omega_{a,b} \circ \Psi_{a,b}^2(\widetilde{x},\widetilde{y})=\Gamma_{a^2}(\widetilde{X},\widetilde{Y}) = \Gamma_{a^2}(X,Y) .
\end{equation*}
\end{proof}

For each natural number $ n $ let us define 
\begin{equation}\label{defP1n}
\mathcal{P}_{1,n}= \left\{ (a,b) \in \mathcal{P}_1: \ a \in \left( 2^{\frac{1}{2^{n+1}}} , 2^{\frac{1}{2^n}} \right] \right\}.
\end{equation}
At this point, we can state the following result, whose proof is an easy consequence of Lemma \ref{lemgammat}, Lemma \ref{leminvdelta} and Lemma \ref{lemconjdelta}.

\begin{proposition}\label{procoedel}
For every $ n \in \mathbb{N} $ and for every $ (a,b) \in \mathcal{P}_{1,n} $ the map $ \Psi_{a,b} $ exhibits $ 2^{n-1} $ strongly topologically mixing strange attractors.
\end{proposition}

% ------------------------------ SUBECTION ------------------------------

\subsection{The region $ \mathcal{P}_2 $}\label{subsecP2}

Let us recall the set $ \mathcal{R}_1=\Omega(C_1,D_1,E_1,F_1) $ introduced in \eqref{defR1}. Now, we define the family of maps $ \mathbb{F}_2 = \{ \Psi_{a,b} \in \mathbb{F} : \; (a,b) \in \mathcal{P}_2 \} $ where, see Figure \ref{figsetP2},
\begin{equation}\label{defFpi}
\mathcal{P}_2 = \left\{ (a,b) \in \mathcal{P} : \; b \geq \frac{2+a}{1+a} \right\}.
\end{equation}
Then, $ (a,b) \in \mathcal{P}_2 $ if and only if $ D_1 $ belongs to $ \mathcal{T}_{1}^+ .$

Let us compute 
\begin{eqnarray}
D_2 & = & \Psi_{a,b}(D_1) = (a(a+b-ab),a(-2+a^2+b+ab-a^2b)) , \nonumber \\
D_3 & = & \Psi_{a,b}(D_2) = \nonumber \\
& & (a(2a-a^3+b-ab-a^2b+a^3b),a(-2+a^2+b+ab-a^2b)). \label{defD3}
\end{eqnarray} 
These points respectively belong to $ \mathcal{T}_1^+ $ and $ \mathcal{T}_0^+ .$ Let 
\begin{equation*}
N = \left(1,a(-2+a^2+b+ab-a^2b) \right)
\end{equation*} 
be the intersection between $ \overline{D_2D_3} $ and $ \mathcal{C} .$ We define the invariant set, see Figure \ref{figRpi}, 
\begin{equation}\label{defRpi}
\mathcal{R}_{a,b} = \Omega(D_1,D_2,D_3,N_1)
\end{equation}
being 
\begin{equation*}
N_1 = \Psi_{a,b}(N) =  \left( a(2 a + b - ab - a^2 b - a^3 + a^3 b), a(b-1) \right).
\end{equation*}
The aim of this section is to prove that for every $ (a,b) \in \mathcal{P}_2 $ there exists certain domain where $ \Psi_{a,b}^2 $ may be conjugate to a direct product of two one-dimensional tent maps.

\begin{figure}[!ht]
\centering
\begin{minipage}{1 \linewidth}
\centering
\includegraphics[width=0.75\textwidth]{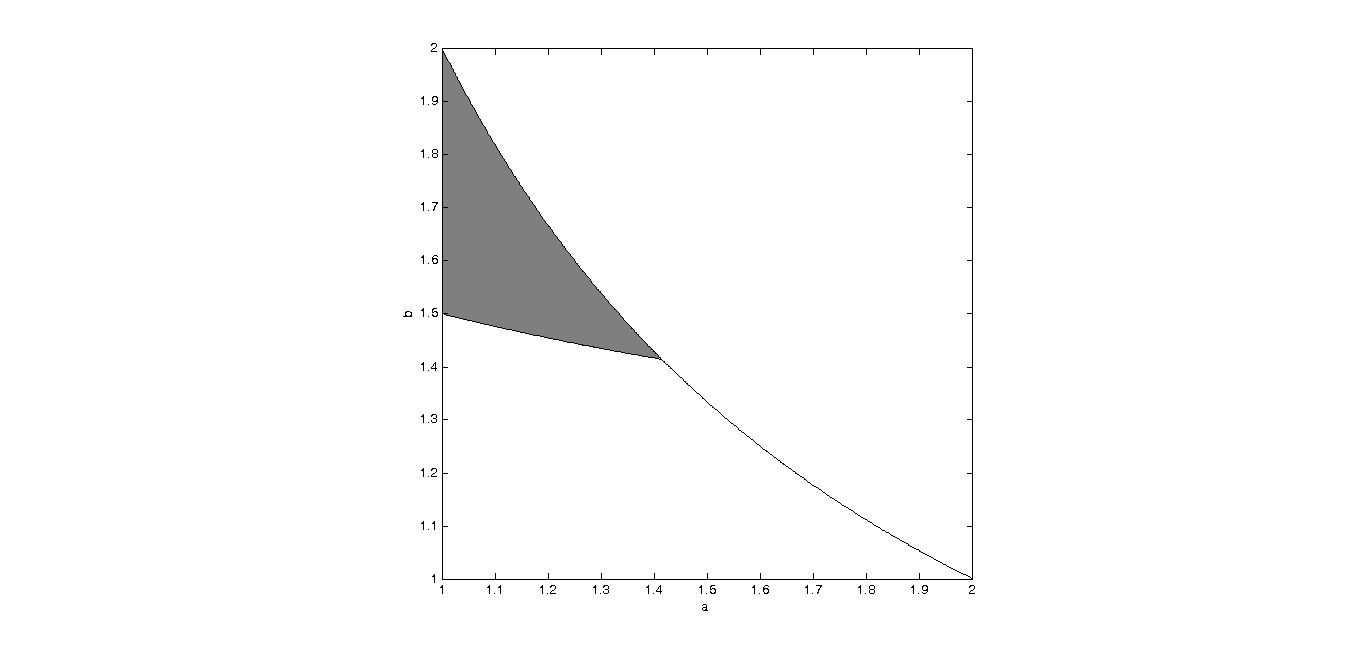}
\caption{The set of parameters $ \mathcal{P}_2 .$}
\label{figsetP2}
\end{minipage}
\end{figure}

% --------------- FIGURE figSab ------------------------
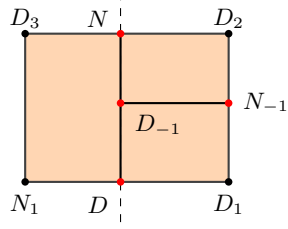
\begin{figure}[!ht]
\centering
\begin{minipage}{0.9 \linewidth}
\centering
{\small
\begin{tikzpicture}[xscale=10,yscale=10]

\fill[mycolor,thick] (1.1437,0.6135) -- (1.1437,0.8107) -- (0.8732,0.8107) -- (0.8732,0.6135); 
\draw[darkgray,thick] (1.1437,0.6135) -- (1.1437,0.8107) -- (0.8732,0.8107) -- (0.8732,0.6135) -- (1.1437,0.6135); 

\draw [black,dashed] (1.0000,0.5600) -- (1.0000,0.8600);
\draw [black,thick] (1.0000,0.6135) -- (1.0000,0.8107);
\draw [black,thick] (1.0000,0.7183) -- (1.1437,0.7183); 
%\draw [black,thick,dashed] (0.7236,0.7236) -- (1.0000,0.2764);
%\draw [black,thick,dashed] (1.0000,0.2764) -- (1.2764,0.7236);

\fill[red] (1.0000,0.6135) circle (0.05mm); 
\node[black] at  (1.0000-0.03,0.6135-0.03) {$ D $}; 
\fill[black] (1.1437,0.6135) circle (0.05mm); 
\node[black] at  (1.1437,0.6135-0.03) {$ D_1 $}; 
\fill[black] (1.1437,0.8107) circle (0.05mm); 
\node[black] at  (1.1437,0.8107+0.02) {$ D_2 $}; 
\fill[black] (0.8732,0.8107) circle (0.05mm); 
\node[black] at  (0.8732,0.8107+0.02) {$ D_3 $}; 
\fill[red] (1.0000,0.8107) circle (0.05mm); 
\node[black] at  (1.0000-0.03,0.8107+0.02) {$ N $}; 
\fill[black] (0.8732,0.6135) circle (0.05mm); 
\node[black] at  (0.8732,0.6135-0.03) {$ N_1 $}; 

\fill[red] (1.0000,0.7183) circle (0.05mm); 
\node[black] at  (1.0000+0.05,0.7183-0.03) {$ D_{-1} $}; 
\fill[red] (1.1437,0.7183) circle (0.05mm); 
\node[black] at  (1.1437+0.05,0.7183) {$ N_{-1} $};

\end{tikzpicture}
}
\caption{The set $ \mathcal{R}_{a,b} .$}
\label{figRpi}
\end{minipage}
\end{figure}
% --------------- END FIGURE figSab ---------------------

\begin{lemma}\label{leminvpi}
For every $ (a,b) \in \mathcal{P}_2 $ the set $ \mathcal{R}_{a,b} $ defined in \eqref{defRpi} is strictly invariant by $ \Psi_{a,b} .$
\end{lemma}

\begin{proof}
We only need to prove that $ D_3 $ belongs to the segment $ \overline{ND_2^{\prime}} ,$ where $ D_2^{\prime} $ is the symmetric point of $ D_2 $  with respect to $ \mathcal{C} .$ 

It is easy to check that the abscise of $ D_2^{\prime} $ is 
\begin{equation*} 
x_2^{\prime} = 2-a^2-ab+a^2b.
\end{equation*}
By comparing $ x_2^{\prime} $ to the abscise of $ D_3 $ (see \eqref{defD3}) one may check that $ D_3 $ belongs to the segment $ \overline{ND_2^{\prime}} $ for every $ (a,b) \in \mathcal{P}_2 .$
\end{proof}

\begin{remark}
As in the previous case, see Lemma \ref{leminvdelta}, we could prove that the set $R_{a,b}$ (see Figure \ref{figRpi}) captures any orbit starting from the interior of $\mathcal{T}.$ However, the proof is not so easy as the one given before. In any case, we want to stress that this kind of proofs is not necessary in order to ensure the existence of strange attractors for $\Psi_{a,b}$ contained in $ \mathcal{R}_{a,b}.$
\end{remark}

\begin{lemma}\label{lemconjpi}
For every $ (a,b) \in \mathcal{P}_2$ there exists an affine change in coordinates such that $ \Psi_{a,b}^2 $ restricted to $ \mathcal{R}_{a,b} $ transforms into the map $ \Gamma_{a^2} $ restricted to $ \mathcal{S}_{a^2} $ (see \eqref{defgamt} and \eqref{defst}).
\end{lemma}

\begin{proof}
Let us first note that $ \Psi_{a,b} $ displays a fixed point in the boundary of $ \mathcal{T}_0^+ $ given by
\begin{equation*}\label{def2perpi}
P_2=(x_2,y_2)=\left(\frac{ab}{1+a},\frac{ab}{1+a}\right).
\end{equation*}

On the other hand, the preimage of $ \mathcal{C} $ in $ \mathcal{T}_1^+ $ is given by 
\begin{equation*}\label{defprelpi}
\mathcal{C}^{-1} = \{ (x,y) \in \mathcal{T}_1^+ : \ y = d \}
\end{equation*}
where $ d=a^{-1}(ab-1) .$ Moreover, $ \mathcal{R}_{a,b} \cap \mathcal{C}^{-1} = \overline{D_{-1}N_{-1}} $ being $ D_{-1} \in \mathcal{C} $ and $ N_{-1} \in \mathcal{T}_1^{+} $ such that $ \Psi_{a,b}(D_{-1})=D $ and $ \Psi_{a,b}(N_{-1})=N .$ Finally, $ \mathcal{R}_{a,b} \cap \mathcal{C} = \overline{DN} .$

Let us consider the change in coordinates 
\begin{equation}\label{defome2}
(X,Y)=\tau_{a,b}(x,y)=\left( \frac{x-x_2}{1-x_2},\frac{y_2-y}{y_2-d} \right).
\end{equation}

It is a simple calculation to check that 
\begin{eqnarray*}
\tau_{a,b}(D_1) & = & (a^2,a^2) \; , \\ 
\tau_{a,b}(D_2) & = & (a^2,a^2(2-a^2)) \; , \\
\tau_{a,b}(D_3) & = & (a^2(2-a^2),a^2(2-a^2)) \; , \\ 
\tau_{a,b}(N_1) & = & (a^2(2-a^2),a^2) \; .
\end{eqnarray*}
Therefore, $ \tau_{a,b}(\mathcal{R}_{a,b})= \mathcal{S}_{a^2} $ (see \eqref{defst}). Moreover,
\begin{eqnarray*}
\tau_{a,b}(\overline{DN}) & = & \{ (X,Y) \in \mathcal{S}_{a^2} : \ X=1 \} \; , \\
\tau_{a,b}(\overline{D_{-1}N_{-1}}) & = & \{ (X,Y) \in \mathcal{S}_{a^2} : \ X \geq 1 , \; Y=1 \} .
\end{eqnarray*}

Let us consider a point $ (x,y) \in \Omega(D,D_1,N_{-1},D_{-1}).$ On one hand, 
\begin{equation*}
\Psi_{a,b}(x,y)=(a(b-y),a(b-2+x)) \in \Omega(D_1,D_2,N,D) \subset \mathcal{T}_1^+
\end{equation*}
and therefore 
\begin{equation*}
\Psi_{a,b}^2(x,y)=(a(b-a(b-2+x)),a(b-2+a(b-y)).
\end{equation*}
On the other hand, $ (X,Y)=\tau_{a,b}(x,y) $ satisfies $ X \geq 1 $ and $ Y \geq 1 $ and, consequently,
\begin{equation*}
\Gamma_{a^2}(X,Y)=(a^2(2-X),a^2(2-Y)).
\end{equation*}
Then, one may check that $ \tau_{a,b} \circ \Psi_{a,b}^2(x,y) = \Gamma_{a^2}(X,Y) .$ % \Gamma_{a^2} \circ \tau_{a,b}(x,y)

Now, let us consider a point $ (x,y) \in \Omega(D_2,N,D_{-1},N_{-1}) $ and let 
\begin{equation*}
(\widetilde{x},\widetilde{y})=(x,2d-y)
\end{equation*}
be the symmetric point of $ (x,y) $ with respect to $ \mathcal{C}^{-1} .$ Then, $ \Psi_{a,b}^2(x,y)= \Psi_{a,b}^2(\widetilde{x},\widetilde{y}). $ Furthermore, $ (X,Y)=\tau_{a,b}(x,y) $ and $ (\widetilde{X},\widetilde{Y})=\tau_{a,b}(\widetilde{x},\widetilde{y}) $ are symmetric points in $ \mathcal{S}_{a^2} $ with respect to $ Y = 1 $ and therefore
\begin{equation*}
\tau_{a,b} \circ \Psi_{a,b}^2(x,y)= \tau_{a,b} \circ \Psi_{a,b}^2(\widetilde{x},\widetilde{y})=\Gamma_{a^2}(\widetilde{X},\widetilde{Y}) = \Gamma_{a^2}(X,Y) .
\end{equation*}

Finally, let us consider a point $ (x,y) \in \Omega(D,N,D_3,N_1) $ and let
\begin{equation*}
(\widetilde{x},\widetilde{y})=(2-x,y)
\end{equation*}
be the symmetric point of $ (x,y) $ with respect to $ \mathcal{C} .$ Then, $ \Psi_{a,b}^2(x,y)= \Psi_{a,b}^2(\widetilde{x},\widetilde{y}). $ Moreover, $ (X,Y)=\tau_{a,b}(x,y) $ and $ (\widetilde{X},\widetilde{Y})=\tau_{a,b}(\widetilde{x},\widetilde{y}) $ are symmetric points in $ \mathcal{S}_{a^2} $ with respect to $ X = 1 $ and therefore
\begin{equation*}
\tau_{a,b} \circ \Psi_{a,b}^2(x,y)= \tau_{a,b} \circ \Psi_{a,b}^2(\widetilde{x},\widetilde{y})=\Gamma_{a^2}(\widetilde{X},\widetilde{Y}) = \Gamma_{a^2}(X,Y) .
\end{equation*}

\end{proof}

For each natural number $ n $ let us define 
\begin{equation}\label{defP2n}
\mathcal{P}_{2,n}= \left\{ (a,b) \in \mathcal{P}_2 : \ a \in \left(2^{\frac{1}{2^{n+1}}},2^{\frac{1}{2^n}} \right] \right\}.
\end{equation}
As a consequence of Lemma \ref{lemgammat}, Lemma \ref{leminvpi} and Lemma \ref{lemconjpi} we have the following result.

\begin{proposition}\label{procoepi}
For every $ n \in \mathbb{N} $ and for every $ (a,b) \in \mathcal{P}_{2,n} $ the map $ \Psi_{a,b} $ exhibits $ 2^{n-1} $ strongly topologically mixing strange attractors.
\end{proposition}

% ------------------------------ SECTION ------------------------------

\section{Renormalization scheme. Proof of Theorem \ref{tmaa}}\label{secren} \

In \cite{PuRoVi} the authors studied a one-parameter family of \textit{EBMs} $ \{ \Lambda_t \}_t $ (see also \eqref{Lambda}). In particular, they proved that there exist three intervals of parameters $ I_3 \subset I_2 \subset I_1 $ such that $ \Lambda_t $ is a $ n $ times renormalizable \textit{EBM} for every $ t \in I_n , \; n=1,2,3 $ (see Main Theorem in \cite{PuRoVi}). Furthermore, for every $ t \in I_n , \; n=1,2,3, $ the map $ \Lambda_t $ is a $ n $ times renormalizable \textit{EBM} in $ 2^{n-1} $ different restrictive domains. In any case, the renormalization scheme requires a change in coordinates in certain triangles $ \Delta_{a,b} $ and $ \Pi_{a,b} $ which are invariant for some power of $ \Lambda_t .$

The aim of this section is to extend this renormalization process to the two-parameter family $ \mathbb{F} .$ To this end, we shall find a set of parameters $ \mathcal{P}_3 \subset \mathcal{P} $ such that, if $ (a,b) \in \mathcal{P}_3 ,$ then $ \Psi_{a,b} $ is renormalizable in $ \mathbb{F} .$ Namely, for each $ (a,b) \in \mathcal{P}_3 ,$ we shall prove that the restriction of $ \Psi_{a,b}^4 $ to each one of two different restrictive domains is conjugate by means of an affine change in coordinates to a \textit{EBM} which belongs to $ \mathbb{F} .$
% ------------------------------ SUBECTION ------------------------------

\subsection{The two restrictive domains}

Let us consider $ \Psi_{a,b} \in \mathbb{F} $ and recall that $ \Psi_{a,b} $ displays a unique fixed point $ P_{a,b}=(x_{a,b},y_{a,b}) \in int(\mathcal{T}) $ (see \eqref{defPab}). Let us denote by $ P^{\prime} $ the symmetric point of $ P_{a,b} $ with respect to $ \mathcal{C} .$ We shall denote by $ \mathcal{L}(Q,m)$ the line passing through a point $Q\in \mathbb{R}^2$ with finite slope $m$ and by $ \mathcal{L}(Q,\infty)$ the vertical line passing through $ Q .$ In this way, we denote by $ Q ,$ $ M $ and $ H_1 $ the intersection between $ \mathcal{C} $ and the lines $ \mathcal{L}(P_{a,b},-1) ,$ $ \mathcal{L}(P_{a,b},0) $ and $ \mathcal{L}(P_{a,b},+1)$ respectively. At this point, we define the triangles $ \Delta_{a,b} $ and $ \Pi_{a,b} ,$ see Figure \ref{figdelpi}, given by
\begin{equation}\label{defdeltapi}
\Delta_{a,b}=\Omega(P_{a,b},P^{\prime},Q) \; , \; \Pi_{a,b}=\Omega(P_{a,b},P^{\prime},H_1) \;.
\end{equation}

% --------------- FIGURE figdelpi ------------------------
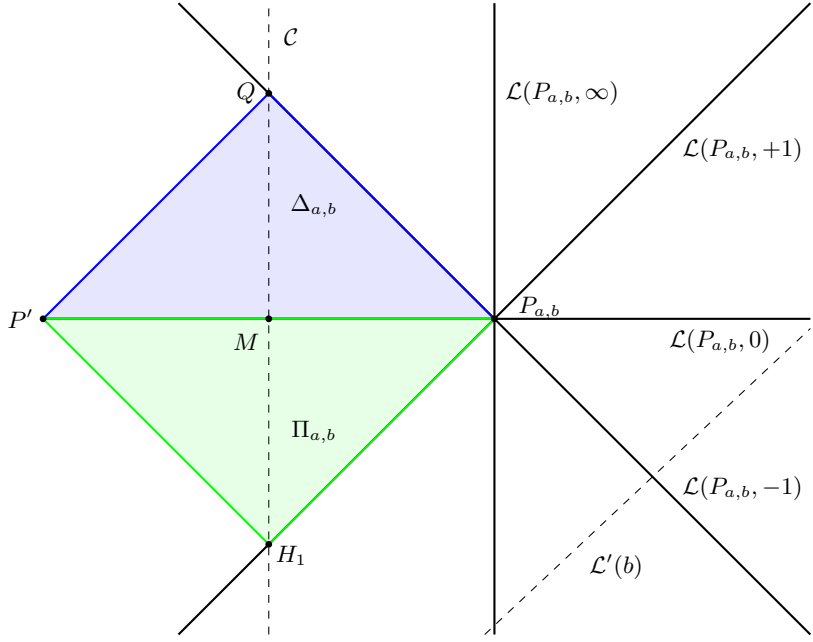
\begin{figure}[!ht]
\centering
\begin{minipage}{0.9 \linewidth}
\centering
{\small
\begin{tikzpicture}[xscale=3,yscale=3]

\draw[black,dashed] (-1.0000,-1.4000) -- (-1.0000,1.4000); 
\node[black] at  (-1.0000+0.1,1.2500) {$ \mathcal{C} $ };
\draw[black,dashed] (-0.0430,-1.4000) -- (1.4000,-0.0430);
\node[black] at  (0.54,-1.1) {$ \mathcal{L}^{\prime}(b) $ }; 
\draw[black,thick] (-2.0000,0.0000) -- (1.4000,0.0000); 
\node[black] at  (1.0000,0.0000-0.1) {$ \mathcal{L}(P_{a,b},0) $};
\draw[black,thick] (-1.4000,-1.4000) -- (1.4000,1.4000); 
\node[black] at  (1.0000+0.1,1.0000-0.25) {$ \mathcal{L}(P_{a,b},+1) $};
\draw[black,thick] (-1.4000,1.4000) -- (1.4000,-1.4000); 
\node[black] at  (1.0000+0.1,-1.0000+0.25) {$ \mathcal{L}(P_{a,b},-1) $};
\draw[black,thick] (0.0000,-1.4000) -- (0.0000,1.4000); 
\node[black] at  (0.0000+0.30,1.0000) {$ \mathcal{L}(P_{a,b},\infty) $};

\fill[blue,opacity=0.10] (0.0000,0.0000) -- (-2.0000,0.0000) -- (-1.0000,1.0000); 
\draw[blue,thick] (0.0000,0.0000) -- (-2.0000,0.0000) -- (-1.0000,1.0000) -- (0.0000,0.0000);
\node[black] at  (-0.8000,0.5000) {$ \Delta_{a,b} $};

\fill[green,opacity=0.10] (0.0000,0.0000) -- (-2.0000,0.0000) -- (-1.0000,-1.0000); 
\draw[green,thick] (0.0000,0.0000) -- (-2.0000,0.0000) -- (-1.0000,-1.0000) -- (0.0000,0.0000); 
\node[black] at  (-0.8000,-0.5000) {$ \Pi_{a,b} $};

\fill[black] (0.0000,0.0000) circle (0.15mm); 
\node[black] at  (0.0000+0.20,0.0000+0.05) {$ P_{a,b} $ }; 
\fill[black] (-2.0000,0.0000) circle (0.15mm); 
\node[black] at  (-2.0000-0.1,0.0000) {$ P^{\prime} $ }; 
\fill[black] (-1.0000,1.0000) circle (0.15mm); 
\node[black] at  (-1.0000-0.1,1.0000) {$ Q $ }; 
\fill[black] (-1.0000,0.0000) circle (0.15mm); 
\node[black] at  (-1.0000-0.1,0.0000-0.1) {$ M $ }; 
\fill[black] (-1.0000,-1.0000) circle (0.15mm); 
\node[black] at  (-1.0000+0.1,-1.0000-0.05) {$ H_1 $ }; 
\end{tikzpicture}
}
\caption{The sets $ \Delta_{a,b} $ and $ \Pi_{a,b} .$ }
\label{figdelpi}
\end{minipage}
\end{figure}
%-------------------------------------------------

Let us introduce the change in coordinates
\begin{equation}\label{defphi}
(X,Y)=\phi_{a,b}(x,y)=\left( \frac{x-x_{a,b}}{x_{a,b}-1} , \frac{y-y_{a,b}}{x_{a,b}-1} \right).
\end{equation}
In new coordinates, $ P_{a,b} $ transforms into $ \mathcal{O} $ and the distance between $ P_{a,b} $ and $ \mathcal{C} $ is one. Now, the critical segments are given by 
\begin{eqnarray*}
\mathcal{C} & = & \{ (X,Y) \in \phi_{a,b}(\mathcal{T}) : \; X=-1\} \; , \\
\mathcal{L}(b) & = & \{ (X,Y) \in \phi_{a,b}(\mathcal{T}_0) : \; X+Y=2-\gamma_{a,b}\} \; , \\
\mathcal{L}^\prime(b) & = & \{ (X,Y) \in \phi_{a,b}(\mathcal{T}_1) : \; X-Y=\gamma_{a,b}\} \; ,
\end{eqnarray*}
being 
\begin{equation}\label{defgamma}
\gamma_{a,b}= \frac{ab+b-2}{-ab+a+1} \; .
\end{equation}
Note that $ \gamma_{a,b} > 0 $ and it is easy to check that
\begin{equation*}
P^{\prime}=(-2,0) \; , \; Q=(-1,1) \; , \; M=(-1,0) \; , \; H_1=(-1,-1) \; .
\end{equation*}

\begin{remark}\label{remaM4}
In new coordinates $ (X,Y) ,$ while the orbit of a point does not leave the region 
\begin{equation*}
\mathcal{T}_1^+ = \{(X,Y) \in \mathcal{T} : X \geq -1 \mbox{ , } X-Y \leq \gamma_{a,b} \} \; ,
\end{equation*} 
one may easily calculate its iterates by $ \Psi_{a,b}.$ In fact, the action of $ \Psi_{a,b} $ on this region consists in a rotation (centered at $P_{a,b}$) by an angle $ \frac{\pi}{2} $ and an expansion by a factor $ a .$ 
\end{remark}

% ------------------------------ SUBECTION ------------------------------

\subsection{Renormalization in $ \Delta_{a,b} $}

In order to simplify the notation, we shall take coordinates $ (X,Y).$ Then, 
\begin{equation*}
P_{a,b}=P=(0,0) \; , \; \Delta_{a,b}=\Delta=\Omega(P,P^{\prime},Q).
\end{equation*}  
For any point $ A \in \mathbb{R}^2 $ we shall write $ A_1=\Psi_{a,b}(A) $ and, in general, $A_i=\Psi_{a,b}(A_{i-1}).$ 

From now on, it will be very useful Figure \ref{figdelta0r2}. If we denote by $\Delta_{i}=\Psi_{a,b}^i(\Delta), \ i=0 \ldots 4 ,$ and recalling that $ \Psi_{a,b} $ is symmetric with respect to the critical lines one has 
\begin{equation*}
\Delta_1=\Psi_{a,b}(\Delta)=\Psi_{a,b}(\Omega(P,Q,M))=\Omega(P,Q_1,M_1) 
\end{equation*}
where $ Q_1=(-a,-a) $ and $ M_1=(0,-a).$

Since $ a > 1, $ the critical line $ \mathcal{C} $ always cuts the interior of $ \Delta_1 $ in points
\begin{equation*}
H_1=(-1,-1) \mbox{ , } K_1=(-1,-a) \; .
\end{equation*}
We shall impose that $ M_1 $ belongs to $ \mathcal{T}_1^+ ,$ i.e., $ a \leq \gamma_{a,b} $ or, equivalently,
\begin{equation}\label{condel1} 
b \geq \frac{2 + a + a^2}{1 + a + a^2}
\end{equation} % 

Now, 
\begin{equation*}
\Delta_2=\Psi_{a,b}(\Omega(P,H_1,K_1,M_1))=\Omega(P,H_2,K_2,M_2) \; .
\end{equation*}
We shall impose that $ H_2=(a,-a) $ and $ M_2=(a^2,0) $ belong to the regions $ \mathcal{T}_1^- $ and $ \mathcal{T}_1^+ $ respectively. These conditions are respectively given by $ 2a \geq \gamma_{a,b} $ and $ a^2 \leq \gamma_{a,b} $ or, equivalently,
\begin{equation}\label{condel2}
b \leq \frac{(2(1 + a + a^2)}{1 + a + 2 a^2}
\end{equation}
and 
\begin{equation}\label{condel3}
b \geq \frac{2 + a^2 + a^3}{1 + a + a^3}.
\end{equation}
In this case, the critical line $ \mathcal{L}^{\prime}(b) $ cuts the interior of $ \Delta_2 $ at points $ \widetilde{H}_2 $ and $ \widetilde{K}_2 $ given by 
\begin{equation*} 
\widetilde{H}_2 = \left( \frac{\gamma_{a,b}}{2},-\frac{\gamma_{a,b}}{2} \right) \mbox{ , } \widetilde{K}_2 =(a^2,a^2-\gamma_{a,b}) \; .
\end{equation*}

In this way, 
\begin{equation*}
\Delta_3=\Psi_{a,b}(\Omega(P, \widetilde{H}_2 , \widetilde{K}_2 , M_2))=\Omega(P, \widetilde{H}_3 , \widetilde{K}_3 , M_3) \; .
\end{equation*}
Since $\Delta_3$ does not intersect the critical set, then
\begin{equation*}
\Delta_4=\Psi_{a,b}(\Omega(P, \widetilde{H}_3 , \widetilde{K}_3 , M_3))=\Omega(P, \widetilde{H}_4 , \widetilde{K}_4 , M_4) \; .
\end{equation*}
being $ \widetilde{H}_4=(-a^2\frac{\gamma_{a,b}}{2},a^2\frac{\gamma_{a,b}}{2}) $ and $ M_4=(-a^4,0) .$ 

In order to determine the set of parameters for which $ \Delta_4=\Psi_{a,b}^4(\Delta) \subset \Delta,$ we need to impose that $ \widetilde{H}_4 \in \overline{PQ} ,$ i.e., $ a^2 \gamma_{a,b} \leq 2 $ or, equivalently,
\begin{equation}\label{condel4}
b \leq \frac{2(1 + a + a^2)}{a (2 + a + a^2)} \; .
\end{equation}

% --------------- FIGURE figre ------------------------
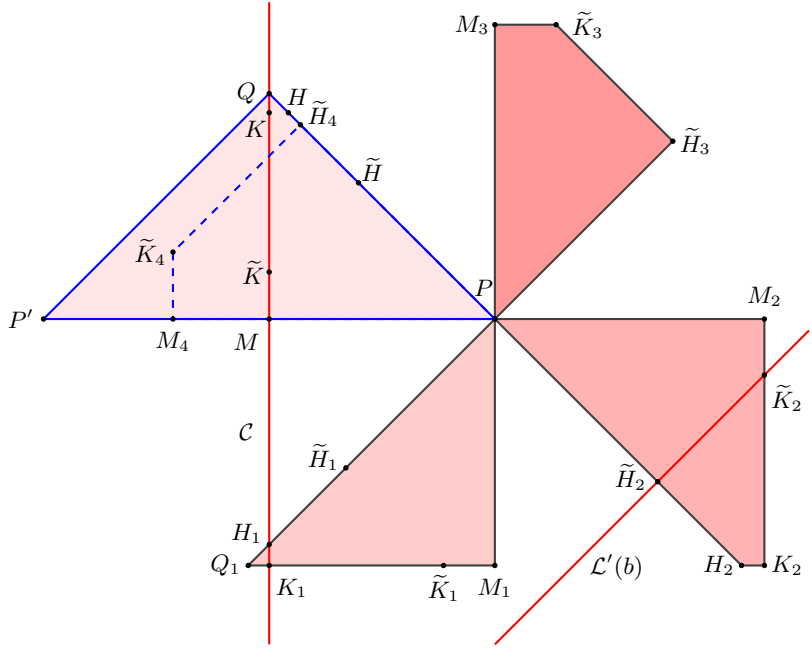
\begin{figure}[!ht]
\centering
\begin{minipage}{0.9 \linewidth}
\centering
{\small
\begin{tikzpicture}[xscale=3,yscale=3]

\fill[red,opacity=0.10] (0.0000,0.0000) -- (-2.0000,0.0000) -- (-1.0000,1.0000); 
\draw[red,thick] (-1.0000,-1.4430) -- (-1.0000,1.4053); 
\draw[red,thick] (1.3919,-0.0510) -- (0.0000,-1.4430);
\draw[blue,thick] (0.0000,0.0000) -- (-2.0000,0.0000) -- (-1.0000,1.0000) -- (0.0000,0.0000); 
\draw[blue,thick,dashed] (0.0000,0.0000) -- (-1.4265,0.0000) -- (-1.4265,0.2969) -- (-0.8617,0.8617) -- (0.0000,0.0000); 

\fill[black] (-2.0000,0.0000) circle (0.12mm); 
\node[black] at  (-2.0000-0.1,0.0000) {$ P^{\prime} $ }; 
\fill[black] (-1.0000,1.0000) circle (0.12mm); 
\node[black] at  (-1.0000-0.1,1.0000) {$ Q $ }; 
\fill[black] (-1.0000,0.0000) circle (0.12mm); 
\node[black] at  (-1.0000-0.1,0.0000-0.1) {$ M $ }; 
\fill[black] (-0.9150,0.9150) circle (0.12mm); 
\node[black] at  (-0.9150+0.04,0.9150+0.07) {$ H $ }; 
\fill[black] (-1.0000,0.9150) circle (0.12mm); 
\node[black] at  (-1.0000-0.06,0.9150-0.07) {$ K $ }; 
\fill[black] (-0.6041,0.6041) circle (0.12mm); 
\node[black] at  (-0.6041+0.06,0.6041+0.06) {$ \widetilde{H} $ }; 
\fill[black] (-1.0000,0.2082) circle (0.12mm); 
\node[black] at  (-1.0000-0.07,0.2082) {$ \widetilde{K} $ }; 

\fill [red,opacity=0.20] (0.0000,0.0000) -- (0.0000,-1.0929) -- (-1.0929,-1.0929); 
\draw [darkgray,thick] (0.0000,0.0000) -- (0.0000,-1.0929) -- (-1.0929,-1.0929) -- (0.0000,0.0000); 
\fill[black] (-1.0929,-1.0929) circle (0.12mm); 
\node[black] at  (-1.0929-0.1,-1.0929) {$ Q_1 $ }; 
\fill[black] (0.0000,-1.0929) circle (0.12mm); 
\node[black] at  (0.0000,-1.0929-0.1) {$ M_1 $ }; 
\fill[black] (-1.0000,-1.0000) circle (0.12mm); 
\node[black] at  (-1.0000-0.1,-1.0000+0.05) {$ H_1 $ }; 
\fill[black] (-1.0000,-1.0929) circle (0.12mm); 
\node[black] at  (-1.0000+0.1,-1.0929-0.1) {$ K_1 $ }; 
\fill[black] (-0.6602,-0.6602) circle (0.12mm); 
\node[black] at  (-0.6602-0.1,-0.6602+0.05) {$ \widetilde{H}_1 $ }; 
\fill[black] (-0.2275,-1.0929) circle (0.12mm); 
\node[black] at  (-0.2275,-1.0929-0.1) {$ \widetilde{K}_1 $ }; 

\fill [red,opacity=0.30] (0.0000,0.0000) -- (1.1944,0.0000) -- (1.1944,-1.0929) -- (1.0929,-1.0929);
\draw [darkgray,thick] (0.0000,0.0000) -- (1.1944,0.0000) -- (1.1944,-1.0929) -- (1.0929,-1.0929) -- (0.0000,0.0000); 
\fill[black] (1.1944,0.0000) circle (0.12mm); 
\node[black] at  (1.1944,0.0000+0.1) {$ M_2 $ }; 
\fill[black] (1.0929,-1.0929) circle (0.12mm); 
\node[black] at  (1.0929-0.1,-1.0929) {$ H_2 $ }; 
\fill[black] (1.1944,-1.0929) circle (0.12mm); 
\node[black] at  (1.1944+0.1,-1.0929) {$ K_2 $ }; 
\fill[black] (0.7215,-0.7215) circle (0.12mm); 
\node[black] at  (0.7215-0.12,-0.7215+0.02) {$ \widetilde{H}_2 $ }; 
\fill[black] (1.1944,-0.2486) circle (0.12mm); 
\node[black] at  (1.1944+0.1,-0.2486-0.1) {$ \widetilde{K}_2 $ }; 

\fill [red,opacity=0.40] (0.0000,0.0000) -- (0.0000,1.3053) -- (0.2717,1.3053) -- (0.7885,0.7885) -- (0.0000,0.0000); 
\draw [darkgray,thick] (0.0000,0.0000) -- (0.0000,1.3053) -- (0.2717,1.3053) -- (0.7885,0.7885) -- (0.0000,0.0000); 
\fill[black] (0.0000,1.3053) circle (0.12mm); 
\node[black] at  (0.0000-0.1,1.3053) {$ M_3 $ }; 
\fill[black] (0.7885,0.7885) circle (0.12mm); 
\node[black] at  (0.7885+0.1,0.7885) {$ \widetilde{H}_3 $ }; 
\fill[black] (0.2717,1.3053) circle (0.12mm); 
\node[black] at  (0.2717+0.14,1.3053) {$ \widetilde{K}_3 $ }; 

\fill[black] (-1.4265,0.0000) circle (0.12mm); 
\node[black] at  (-1.4265,0.0000-0.1) {$ M_4 $ }; 
\fill[black] (-0.8617,0.8617) circle (0.12mm); 
\node[black] at  (-0.8617+0.10,0.8617+0.05) {$ \widetilde{H}_4 $ }; 
\fill[black] (-1.4265,0.2969) circle (0.12mm); 
\node[black] at  (-1.4265-0.1,0.2969) {$ \widetilde{K}_4 $ }; 

\fill[black] (0.0000,0.0000) circle (0.12mm); 
\node[black] at  (0.0000-0.05,0.0000+0.14) {$ P $ }; 

\node[black] at  (-1.0000-0.1,-0.5) {$ \mathcal{C} $ }; 
\node[black] at  (0.54,-1.1) {$ \mathcal{L}^{\prime}(b) $ }; 

\end{tikzpicture}
}
\caption{The iterates of $ \Delta_0:$ encircled in blue, the triangle $ \Delta_0 ;$ encircled in a dashed blue line, the set $ \Delta_4 .$}
\label{figdelta0r2}
\end{minipage}
\end{figure}
%-------------------------------------------------

Since $ a > 1 ,$ it is clear that condition \eqref{condel3} implies condition \eqref{condel1}. On the other hand, condition \eqref{condel4} implies condition \eqref{condel2}. Hence, see Figure \ref{figsetPdelta}, we define the set of parameters
\begin{equation}\label{defPdelta}
\mathcal{P}_{\Delta}= \left\{(a,b) \in \mathcal{P} : \; \frac{2 + a^2 + a^3}{1 + a + a^3} \leq b \leq \frac{2(1 + a + a^2)}{a (2 + a + a^2)} \right\}.
\end{equation}

% --------------- FIGURE figsetPdelta ------------------------

\begin{figure}[!ht]
\centering
\begin{minipage}{1 \linewidth}
\centering
\includegraphics[width=0.75\textwidth]{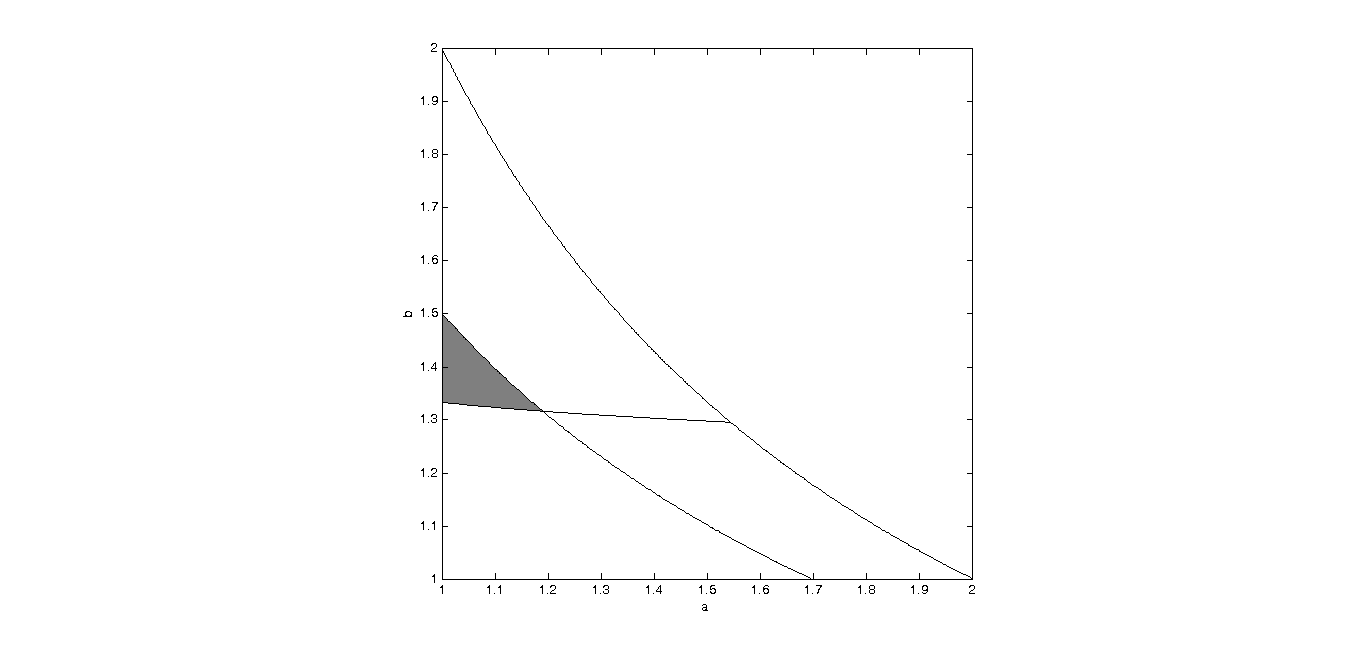}
\caption{The set of parameters $ \mathcal{P}_{\Delta} .$}
\label{figsetPdelta}
\end{minipage}
\end{figure}
%-------------------------------------------------

\begin{proposition}\label{prodel}
The domain $ \Delta_{a,b} $ is invariant by $ \Psi_{a,b}^4 $ if and only if $ (a,b) \in \mathcal{P}_{\Delta} .$ Furthermore, for every $ (a,b) \in \mathcal{P}_{\Delta} $ the map $ \Psi_{a,b} $ is renormalizable in $ \Delta_{a,b} .$ Namely, $ \Psi_{a,b}^4 $ is conjugate by means of an affine change in coordinates $ \widetilde{\phi}_{a,b} $ to the map
\begin{equation}\label{defpsitil}
\widetilde{\Psi}_{a^{\prime},b^{\prime}} = EBM(\mathcal{T},\mathcal{C},\mathcal{H},\mathcal{L}(b^{\prime}),\mathcal{O},B_{a^{\prime}} )
\end{equation}
being $ a^{\prime}=a^4 ,$ $ b^{\prime}=a^{-2}\gamma_{a,b} ,$ $ B_{a^{\prime}} $ the matrix introduced in \eqref{defBa} and $ \mathcal{H} $ the segment in $ \mathcal{T}_0^+ $ given by $ y=a^{-1}.$
\end{proposition}

\begin{proof}
As a consequence of the construction of $ \mathcal{P}_{\Delta} ,$ if $ (a,b) \in \mathcal{P}_{\Delta} $ then $ \Delta_{a,b} $ is invariant by $ \Psi_{a,b}^4 .$ 

On the other hand, if $ b > \frac{2(1 + a + a^2)}{a (2 + a + a^2)} $ then $ \widetilde{H}_4 \notin \overline{PQ} $ and it is evident that $ \Delta_{a,b} $ is not invariant by $ \Psi_{a,b}^4 .$ Finally, if $ b < \frac{2 + a^2 + a^3}{1 + a + a^3} $ then $ a^2 > \gamma_{a,b} $ and thus $ M_2 \in \mathcal{T}_1^- $ and 
\begin{equation*}
M_4=(-a^2\gamma_{a,b},-a^2(a^2-\gamma_{a,b})) 
\end{equation*}
does not belong to $ \Delta_4 .$

Now, we shall prove that $ \Psi_{a,b}^4 $ is conjugate to the map $ \widetilde{\Psi}_{a^{\prime},b^{\prime}} $ given in \eqref{defpsitil}.

Let us define the points $ H $ and $ K $ in $ \Delta_{a,b} $ such that $ \Psi_{a,b}(H)=H_1 $ and $ \Psi_{a,b}(K)=K_1 $ (see Figure \ref{figdelta0r2}). In coordinates $ (X,Y) ,$ these points are given by
\begin{equation*} 
H = \left( -\frac{1}{a},\frac{1}{a} \right) \mbox{ , } K = \left( -1,\frac{1}{a} \right) \; .
\end{equation*}

We may also define the points $ \widetilde{H} $ and $ \widetilde{K} $ in $ \Delta_{a,b} $ such that $ \Psi_{a,b}^2(\widetilde{H})=\widetilde{H}_2 $ and $ \Psi_{a,b}^2(\widetilde{K})=\widetilde{K}_2 .$ These points are given by, see Figure \ref{figdelta0r2},
\begin{equation*} 
\widetilde{H}=\left(-\frac{\gamma_{a,b}}{2a^2},\frac{\gamma_{a,b}}{2a^2} \right) \mbox{ , } \widetilde{K}=(-1,-(1-\frac{\gamma_{a,b}}{a^2})) 
\end{equation*}
in such a way that $ \Delta_4 = \Psi_{a,b}^4(\Delta_0) $ where $ \Delta_0=\Omega(P,\widetilde{H},\widetilde{K},M) .$

Moreover, since $ \Psi_{a,b}^k(\Delta_0) $ does not leave the region 
\begin{equation*} 
\mathcal{T}_1^+ = \{ (X,Y) : \ X > - 1 \mbox{ , } X-Y < \gamma_{a,b} \} 
\end{equation*}
for $ k=0,\ldots,3 $ we have that $ \Psi_{a,b | \Delta_0}^4 $ is the homothecy with expansion rate $ a^4 .$

Therefore, after the change in coordinates $ \widetilde{\phi}_{a,b} = h \circ \phi_{a,b} ,$ where $ h(x,y)=(-X,Y) ,$ the invariant domain $ \Delta_{a,b} $ transforms into our well-known domain $ \mathcal{T} .$ Furthermore, $ \Psi_{a,b}^4 $ coincides with $ EBM(\mathcal{T},\mathcal{C},\mathcal{H},\mathcal{L}(b^{\prime}),\mathcal{O},B_{a^{\prime}}) $ where $ a^{\prime}=a^4 ,$ $ b^{\prime}=a^{-2}\gamma_{a,b} $ and $ \mathcal{H} $ is the segment in $ \mathcal{T}_0^+ $ given by $ y=a^{-1}.$
\end{proof}

\begin{remark}\label{remadel}
Let us note that, given $ (a,b) \in \mathcal{P}_{\Delta} ,$ then $ (a^{\prime},b^{\prime}) \in \mathcal{P} .$
\end{remark}

From the construction of $ \Delta_4 $ it is easy to check that $ \Psi_{a,b}^4(\Delta_4)=\Delta_4 $ (see Figure \ref{figdelta0r2}). Let us denote by $ \widetilde{\Delta}_{a,b} $ the set $ \phi_{a,b}^{-1}(\Delta_4) $ where $ \phi_{a,b} $ is the change in coordinates given in \eqref{defphi}. Then, $ \widetilde{\Delta}_{a,b} \subset \Delta_{a,b} $ satisfies
\begin{equation}\label{defdelp}
\widetilde{\Delta}_{a,b}=\Omega(P_{a,b},  \widetilde{H}_4 , \widetilde{K}_4 , M_4).
\end{equation}

We shall now impose that the point $ \widetilde{H}_4 $ belongs to the segment $ \overline{PH} .$ This is equivalent to $ a^3\gamma_{a,b} < 2 ,$ i.e.,
\begin{equation}\label{condel5} 
b \leq \frac{2(1+a+a^3)}{a(2+a^2+a^3)}.
\end{equation}
Therefore, we define, see Figure \ref{figPdelp}, the set of parameters 
\begin{eqnarray}
\mathcal{P}_{\Delta}^{\prime} & = & \{ (a,b) \in \mathcal{P}_{\Delta} : \; a^3 \gamma_{a,b} < 2 \} \nonumber \\
& = & \left\{(a,b) \in \mathcal{P} : \; \frac{2 + a^2 + a^3}{1 + a + a^3} \leq b \leq \frac{2(1+a+a^3)}{a(2+a^2+a^3)} \right\} .
\label{defPdeltap}
\end{eqnarray}

\begin{figure}[!ht]
\centering
\begin{minipage}{1 \linewidth}
\centering
\includegraphics[width=0.75\textwidth]{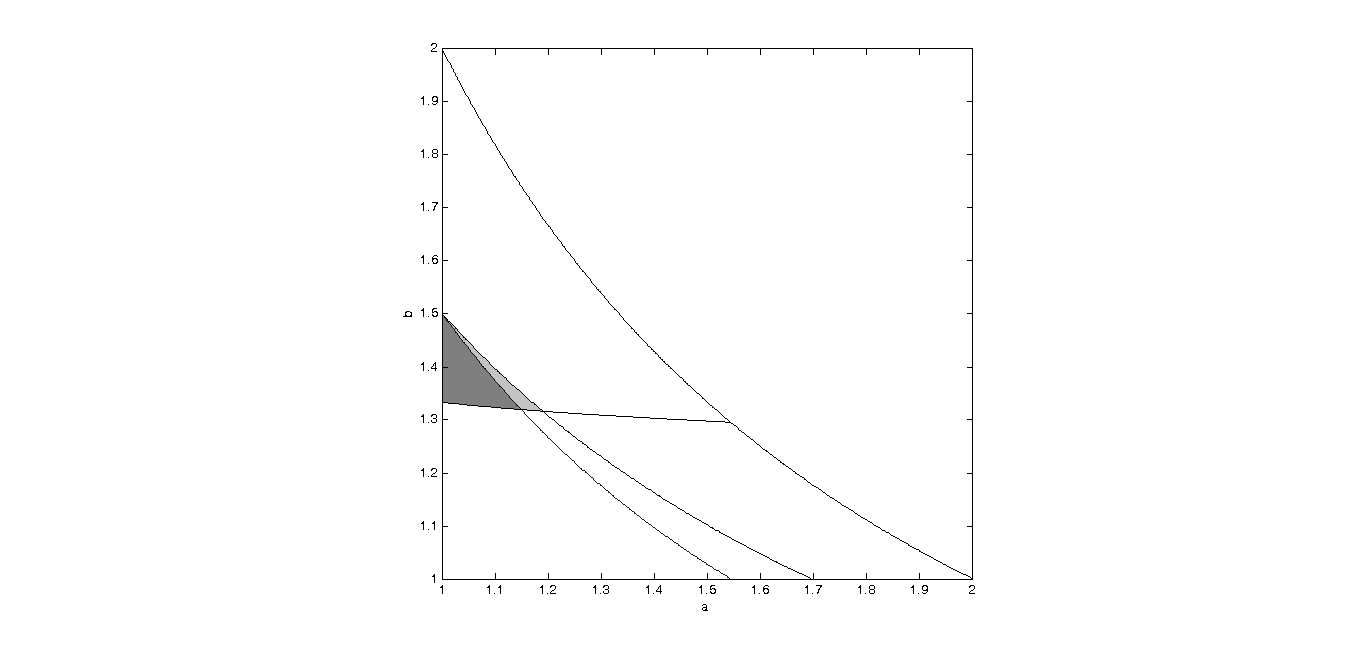}
\caption{The sets of parameters $ \mathcal{P}_{\Delta}^{\prime} $ (dark grey) and $ \mathcal{P}_{\Delta} $ (pale grey).}
\label{figPdelp}
\end{minipage}
\end{figure}

Thus, the following result is consequence of the construction of $ \mathcal{P}_{\Delta}^{\prime} $ and Proposition \ref{prodel}.
 
\begin{proposition}\label{prodelp}
For every $ (a,b) \in \mathcal{P}_{\Delta}^{\prime} $ the domain $ \widetilde{\Delta}_{a,b} $ is invariant by $ \Psi_{a,b}^4 .$ Furthermore, for every $ (a,b) \in \mathcal{P}_{\Delta}^{\prime} $ the map $ \Psi_{a,b} $ is renormalizable in $ \widetilde{\Delta}_{a,b} .$ Namely, $ \Psi_{a,b}^4 $ is conjugate by means of an affine change in coordinates $ \widetilde{\phi}_{a,b} $ to the map
\begin{equation}\label{defpsip}
\Psi_{a^{\prime},b^{\prime}} = EBM(\mathcal{T},\mathcal{C},\mathcal{L}(b^{\prime}),\mathcal{O},B_{a^{\prime}} )
\end{equation}
restricted to $ \Psi_{a^{\prime},b^{\prime}}(\mathcal{T}) ,$ being $ a^{\prime}=a^4 ,$ $ b^{\prime}=a^{-1}\gamma_{a,b} $ and $ B_{a^{\prime}} $ the matrix introduced in \eqref{defBa}. 
\end{proposition}

% ------------------------------ SUBECTION ------------------------------

\subsection{Renormalization in $ \Pi_{a,b} $}

In order to simplify the notation, we shall take coordinates $ (X,Y).$ Then, 
\begin{equation*}
P_{a,b}=P=(0,0) \; , \; \Pi_{a,b}=\Pi=\Omega(P,P^{\prime},H_1).
\end{equation*}  
For any point $ A \in \mathbb{R}^2 $ we still write $ A_1=\Psi_{a,b}(A) $ and $A_i=\Psi_{a,b}(A_{i-1}).$ 

From now on, it will be very useful Figure \ref{figpi0r2}. If we denote by $\Pi_{i}=\Psi_{a,b}^i(\Pi), \ i=0 \ldots 4 ,$ and recalling that $ \Psi_{a,b} $ is symmetric with respect to the critical lines one has 
\begin{equation*}
\Pi_1=\Psi_{a,b}(\Pi)=\Psi_{a,b}(\Omega(P,M,H_1))=\Omega(P,M_1,H_2).
\end{equation*}

As we have done in the previous subsection, let us suppose that $ M_1 \in \mathcal{T}_1^+ $ and $ H_2 \in \mathcal{T}_1^-. $ Equivalently, conditions \eqref{condel1} and \eqref{condel2} hold. We denote by $J_1$ and $ \widetilde{H}_2 $ the intersection between the critical line and the lines $\mathcal{L}(M_1,0)$ and $\mathcal{L}(P,-1)$ respectively. In coordinates $(X,Y)$ these points are given by 
\begin{equation*} 
J_1=(\gamma_{a,b}-a,-a) \mbox{ , } \widetilde{H}_2 = \left( \frac{\gamma_{a,b}}{2},-\frac{\gamma_{a,b}}{2} \right) \; .
\end{equation*}

Now, 
\begin{equation*}
\Pi_2=\Psi_{a,b}(\Omega(P,\widetilde{H}_2,J_1,M_1))=\Omega(P,\widetilde{H}_3,J_2,M_2) \; .
\end{equation*}
Let us now impose that $ \Pi_2 $ does not intersect the critical set. This is equivalent to assume that $ M_2 \in \mathcal{T}_1^+ ,$ i.e., condition \eqref{condel3} holds.

Under this assumption,
\begin{equation*}
\Pi_3=\Psi_{a,b}(\Omega(P,\widetilde{H}_3,J_2,M_2))=\Omega(P, \widetilde{H}_4,J_3,M_3).
\end{equation*} 
Let us now impose that $ \Pi_3 $ does not intersect the critical set. This is equivalent to assume that $ \widetilde{H}_4 \in \overline{PQ} ,$ i.e., condition \eqref{condel4} holds.
In this case,
\begin{equation*}
\Pi_4=\Psi_{a,b}(\Omega(P,\widetilde{H}_4,J_3,M_3))=\Omega(P, \widetilde{H}_5,J_4,M_4) \; .
\end{equation*}

Finally, in order to determine the set of values where $ \Pi_4=\Psi_{a,b}^4(\Pi) \subset \Pi,$ we need to impose that $ \widetilde{H}_5 \in \overline{PH_1} ,$ i.e., $ a^3 \gamma_{a,b} \leq 2 $ or, equivalently,
\begin{eqnarray}\label{conpi1}
b \leq \frac{2(1+a+a^3)}{a(2+a^2+a^3)} \; .
\end{eqnarray}

% --------------- FIGURE onedimtentmapren ------------------------
\begin{figure}[!ht]
\centering
\begin{minipage}{0.9 \linewidth}
\centering
{\small
\begin{tikzpicture}[xscale=3,yscale=3]

\fill[red,opacity=0.10] (0.0000,0.0000) -- (-2.0000,0.0000) -- (-1.0000,-1.0000);
\draw[red,thick] (-1.0000,-1.4430) -- (-1.0000,1.4053); 
\draw[red,thick] (1.3919,-0.0510) -- (0.0000,-1.4430); 
\draw[green,thick] (0.0000,0.0000) -- (-2.0000,0.0000) -- (-1.0000,-1.0000) -- (0.0000,0.0000); 
\draw[green,thick,dashed] (0.0000,0.0000) -- (-1.4265,0.0000) -- (-1.4265,-0.4570) -- (-0.9418,-0.9418) -- (0.0000,0.0000); 
 
\fill[black] (-2.0000,0.0000) circle (0.12mm); 
\node[black] at  (-2.0000-0.1,0.0000) {$ P^{\prime} $ }; 
\fill[black] (-1.0000,-1.0000) circle (0.12mm); 
\node[black] at  (-1.0000-0.1,-1.0000-0.05) {$ H_1 $ }; 
\fill[black] (-0.6602,-0.6602) circle (0.12mm); 
\node[black] at  (-0.6602+0.05,-0.6602-0.1) {$ \widetilde{H}_1 $ }; 
\fill[black] (-1.0000,-0.3204) circle (0.12mm); 
\node[black] at  (-1.0000-0.07,-0.3204) {$ J $ }; 
\fill[black] (-1.0000,0.0000) circle (0.12mm); 
\node[black] at  (-1.0000+0.06,0.0000+0.06) {$ M $ }; 

\fill [red,opacity=0.20] (0.0000,0.0000) -- (0.0000,-1.0929) -- (1.0929,-1.0929);
\draw[darkgray,thick] (0.0000,0.0000) -- (0.0000,-1.0929) -- (1.0929,-1.0929) -- (0.0000,0.0000); 
\fill[black] (1.0929,-1.0929) circle (0.12mm); 
\node[black] at  (1.0929+0.1,-1.0929) {$ H_2 $ }; 
\fill[black] (0.0000,-1.0929) circle (0.12mm); 
\node[black] at  (0.0000,-1.0929-0.1) {$ M_1 $ }; 
\fill[black] (0.7215,-0.7215) circle (0.12mm); 
\node[black] at  (0.7215,-0.7215+0.16) {$ \widetilde{H}_2 $ }; 
\fill[black] (0.3501,-1.0929) circle (0.12mm); 
\node[black] at  (0.3501+0.05,-1.0929-0.1) {$ J_1 $ }; 

\fill [red,opacity=0.30] (0.0000,0.0000) -- (1.1944,0.0000) -- (1.1944,0.3826) -- (0.7885,0.7885);
\draw[darkgray,thick] (0.0000,0.0000) -- (1.1944,0.0000) -- (1.1944,0.3826) -- (0.7885,0.7885) -- (0.0000,0.0000); 
\fill[black] (1.1944,0.0000) circle (0.12mm); 
\node[black] at  (1.1944,0.0000-0.1) {$ M_2 $ }; 
\fill[black] (1.1944,0.3826) circle (0.12mm); 
\node[black] at  (1.1944+0.1,0.3826) {$ J_2 $ }; 
\fill[black] (0.7885,0.7885) circle (0.12mm); 
\node[black] at  (0.7885+0.1,0.7885+0.02) {$ \widetilde{H}_3 $ }; 

\fill [red,opacity=0.40] (0.0000,0.0000) -- (0.0000,1.3053) -- (-0.4182,1.3053) -- (-0.8617,0.8617);
\draw[darkgray,thick] (0.0000,0.0000) -- (0.0000,1.3053) -- (-0.4182,1.3053) -- (-0.8617,0.8617) -- (0.0000,0.0000); 
\fill[black] (0.0000,1.3053) circle (0.12mm); 
\node[black] at  (0.0000+0.1,1.3053) {$ M_3 $ }; 
\fill[black] (-0.4182,1.3053) circle (0.12mm); 
\node[black] at  (-0.4182,1.3053+0.1) {$ J_3 $ }; 
\fill[black] (-0.8617,0.8617) circle (0.12mm); 
\node[black] at  (-0.8617,0.8617-0.1) {$ \widetilde{H}_4 $ }; 

\fill[black] (-1.4265,0.0000) circle (0.12mm); 
\node[black] at  (-1.4265,0.0000+0.1) {$ M_4 $ }; 
\fill[black] (-1.4265,-0.4570) circle (0.12mm); 
\node[black] at  (-1.4265+0.1,-0.4570) {$ J_4 $ }; 
\fill[black] (-0.9418,-0.9418) circle (0.12mm); 
\node[black] at  (-0.9418+0.1,-0.9418-0.02) {$ \widetilde{H}_5 $ }; 

\fill[black] (0.0000,0.0000) circle (0.12mm); 
\node[black] at  (0.0000+0.05,0.0000+0.1) {$ P $ };

\node[black] at  (-1.0000-0.1,0.5) {$ \mathcal{C} $ }; 
\node[black] at  (1.24,-0.4) {$ \mathcal{L}^{\prime}(b) $ }; 
\end{tikzpicture}
}
\caption{The iterates of $ \Pi_0 :$ encircled in green, the triangle $ \Pi_0 ;$ encircled in a dashed green line, the set $ \Pi_4 .$}
\label{figpi0r2}
\end{minipage}
\end{figure}
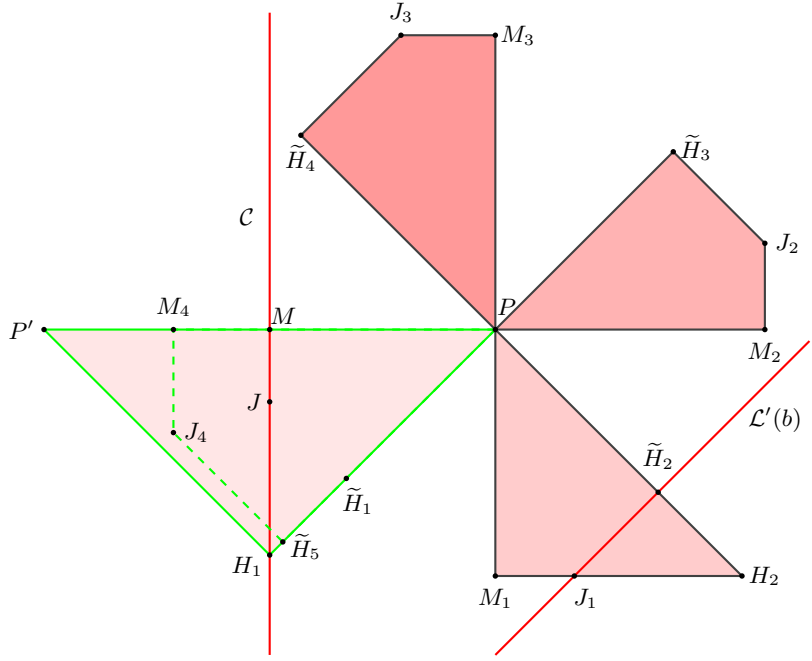
%-------------------------------------------------

Since $ a > 1 ,$ it is clear that condition \eqref{conpi1} implies condition \eqref{condel4}. In this way, we define the set of parameters, see Figure \ref{figsetPpi},
\begin{equation}\label{defPpi}
\mathcal{P}_{\Pi}=\left\{(a,b) \in \mathcal{P} : \frac{2 + a^2 + a^3}{1 + a + a^3} \leq b \leq \frac{2(1+a+a^3)}{a(2+a^2+a^3)} \right\} .
\end{equation}

\begin{proposition}\label{propi}
For every $ (a,b) \in \mathcal{P}_{\Pi} $ the domain $ \Pi_{a,b} $ is invariant by $ \Psi_{a,b}^4 .$ Furthermore, for every $ (a,b) \in \mathcal{P}_{\Pi} $ the map $ \Psi_{a,b} $ is renormalizable in $ \Pi_{a,b} .$ Namely, $ \Psi_{a,b}^4 $ is conjugate by means of an affine change in coordinates $ \widehat{\phi}_{a,b} $ to the map
\begin{equation}\label{defpsipp}
\Psi_{a^{\prime\prime},b^{\prime\prime}} = EBM(\mathcal{T},\mathcal{C},\mathcal{L}(b^{\prime\prime}),\mathcal{O},B_{a^{\prime\prime}} )
\end{equation}
being $ a^{\prime\prime}=a^4 ,$ $ b^{\prime\prime}=a^{-1}\gamma_{a,b} $ and $ B_{a^{\prime\prime}} $ the matrix introduced in \eqref{defBa}.
\end{proposition}

\begin{proof}
As a consequence of the construction of $ \mathcal{P}_{\Pi} ,$ if $ (a,b) \in \mathcal{P}_{\Pi} $ then $ \Pi_{a,b} $ is invariant by $ \Psi_{a,b}^4 .$ 

Now, we shall prove that $ \Psi_{a,b}^4 $ is conjugate to the map $ \Psi_{a^{\prime\prime},b^{\prime\prime}} $ given in \eqref{defpsipp}.

Let us define the points $ \widetilde{H}_1 $ and $ J $ in $ \Pi_{a,b} $ such that $ \Psi_{a,b}(\widetilde{H}_1)=\widetilde{H}_2 $ and $ \Psi_{a,b}(J)=J_1 $ (see Figure \ref{figpi0r2}). In coordinates $ (X,Y) ,$ these points are given by 
\begin{equation*} 
\widetilde{H}_1= \left( -\frac{\gamma_{a,b}}{2a},-\frac{\gamma_{a,b}}{2a} \right) \mbox{ , } J = \left( -1 , 1-\frac{\gamma_{a,b}}{a} \right)  
\end{equation*}
in such a way that $ \Pi_4 = \Psi_{a,b}^4(\Pi_0) $ where $ \Pi_0=\Omega(P,\widetilde{H}_1,J,M) .$

Moreover, since $ \Psi_{a,b}^k(\Pi_0) $ does not leave the region 
\begin{equation*} 
\mathcal{T}_1^+ = \{ (X,Y) : \ X > - 1 \mbox{ , } X-Y < \gamma_{a,b} \} 
\end{equation*}
for $ k=0,\ldots,3 $ we have that $ \Psi_{a,b | \Pi_0}^4 $ is the homothecy with expansion rate $ a^4 .$

Therefore, after the change in coordinates $ \widehat{\phi}_{a,b} = \widehat{h} \circ \phi_{a,b} ,$ where $ h(x,y)=(-X,-Y) ,$ the invariant domain $ \Pi_{a,b} $ transforms into our well-known domain $ \mathcal{T} .$ Furthermore, $ \Psi_{a,b}^4 $ coincides with $ EBM(\mathcal{T},\mathcal{C},\mathcal{L}(b^{\prime\prime}),\mathcal{O},B_{a^{\prime\prime}}) $ where  $ a^{\prime\prime}=a^4 $ and $ b^{\prime\prime}=a^{-1}\gamma_{a,b} .$
\end{proof}

\begin{figure}[!ht]
\centering
\begin{minipage}{1 \linewidth}
\centering
\includegraphics[width=0.75\textwidth]{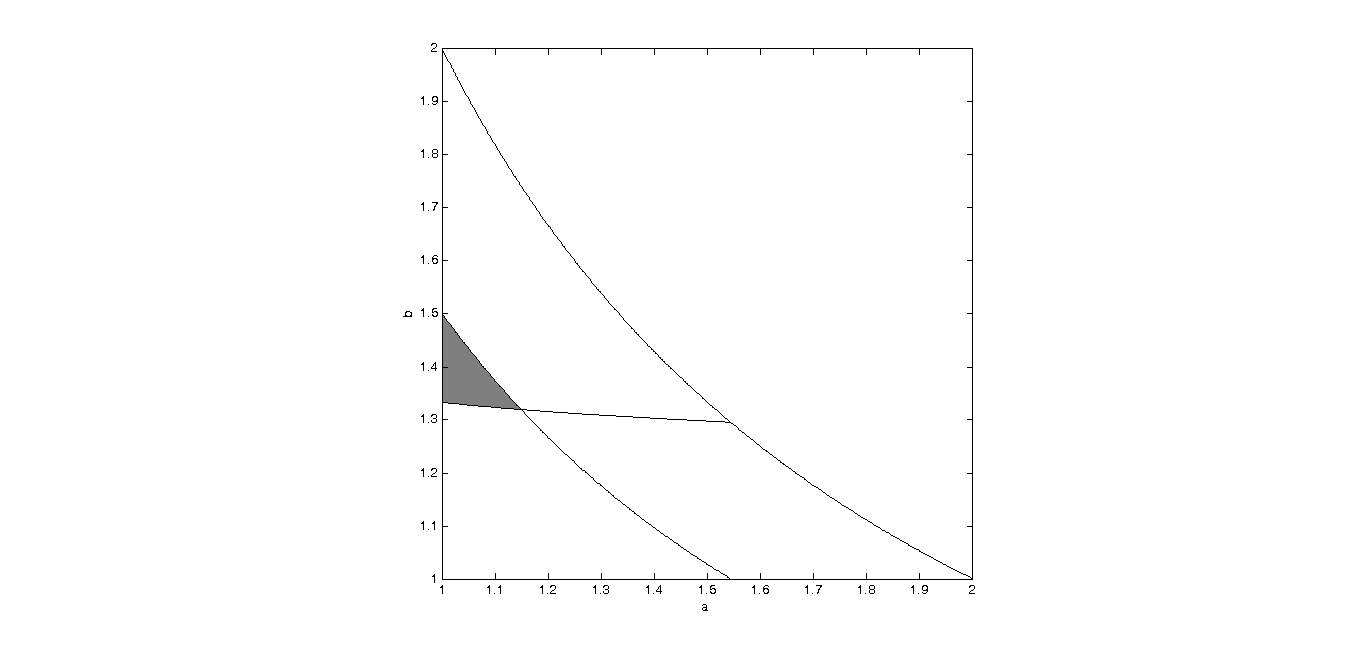}
\caption{The set of parameters $ \mathcal{P}_{\Pi} .$}
\label{figsetPpi}
\end{minipage}
\end{figure}

\begin{remark}\label{remapi}
Let us note that, given $ (a,b) \in \mathcal{P}_{\Pi} ,$ then $ (a^{\prime\prime},b^{\prime\prime}) \in \mathcal{P} .$
\end{remark}

\begin{remark}
Let us remark that condition \eqref{condel3} is not necessary to guarantee the invariance of $ \Pi_{a,b} .$ In fact, it holds that $ \Psi_{a,b}^4(\Pi_{a,b}) \subset \Pi_{a,b} $ if and only if $ (a,b) \in \mathcal{P} $ verifies conditions \eqref{condel1} and \eqref{conpi1}. However, if condition \eqref{condel3} is not satisfied, the map $ \Psi_{a,b} $ is not renormalizable in $ \mathbb{F} :$ the renormalized map is a \textit{EBM} with three folds. 
\end{remark}

% ------------------------------ SUBECTION ------------------------------

\subsection{Simultaneous renormalizations. The region $\mathcal{P}_3$}

To begin with, we note that, see \eqref{defPdeltap} and \eqref{defPpi}, $ \mathcal{P}_{\Pi} = \mathcal{P}_{\Delta}^{\prime}. $ From now on, let us denote this set of parameters by $ \mathcal{P}_3 .$ According to Remark \ref{remadel} and Remark \ref{remapi} we may define the maps 
\begin{eqnarray}
H_{\Delta} & : & (a,b) \in \mathcal{P}_{3} \mapsto H_{\Delta}(a,b)=\left( a^4 , \frac{\gamma_{a,b}}{a^2} \right) \in \mathcal{P} \label{defHdel} \\[1.5ex]
H_{\Pi} & : & (a,b) \in \mathcal{P}_{3} \mapsto H_{\Pi}(a,b)=\left( a^4 , \frac{\gamma_{a,b}}{a} \right) \in \mathcal{P} \label{defHpi}
\end{eqnarray}

As a consequence of Proposition \ref{prodelp} and Proposition \ref{propi} we have the following result.

\begin{theorem}
For every $ (a,b) \in \mathcal{P}_{3} $ the map $ \Psi_{a,b} $ is simultaneously renormalizable in $ \mathbb{F} $ on two different restrictive domains. Namely:
\begin{itemize}
\item[i)] The restriction of $ \Psi_{a,b}^4 $ to $ \widetilde{\Delta}_{a,b} $ is conjugate by means of an affine change in coordinates to $ \Psi_{H_{\Delta}(a,b)} $ restricted to $ \Psi_{H_{\Delta}(a,b)}(\mathcal{T}) .$
\item[ii)] The restriction of $ \Psi_{a,b}^4 $ to $ \Pi_{a,b} $ is conjugate by means of an affine change in coordinates to $ \Psi_{H_{\Pi}(a,b)}.$
\end{itemize}
\end{theorem} 

Let us consider the domain of parameters 
\begin{equation}\label{defpest}
\overline{\mathcal{P}} = \{(a,b) \in [1,2] \times [1,2) : \; ab<2 \}.
\end{equation}
and let $ D $ be an arbitrarily small neighborhood of $ \overline{\mathcal{P}} $ not containing the point $ (1,2) .$ We also assume that $ D $ is chosen in such a way that the maps, see \eqref{defHdel}, \eqref{defHpi} and \eqref{defgamma},
\begin{equation*}
H_{\Delta}(a,b) = \left( a^4,\frac{-2+b+ab}{a^2(1+a-ab)} \right) \mbox{ , } H_{\Pi}(a,b) = \left( a^4,\frac{-2+b+ab}{a(1+a-ab)} \right)
\end{equation*}
are well defined in $ D .$

\begin{proposition}\label{proHdel}
For $H=H_{\Delta}$ it holds that:
\begin{enumerate}
\item Each vertical fiber $ V_{\alpha} = \left\{ (a,b) \in D : \; a = \alpha \right\} $ is sent by $ H $ into the vertical fiber $ V_{\alpha^{4}} $ and the map $ H(\alpha,\cdot) $ is strictly increasing on $ V_{\alpha} $ for every $ \alpha>1 .$
\item $ H $ is an area expanding diffeomorphism on $ D $ having a unique fixed point $ P^{\ast}=(1,\sqrt{2})$. The eigenvalues of $ DH(P^{\ast}) $ are $ \lambda_{1}=4 $ and $ \lambda_{2}=3+2\sqrt{2} $ with eigenvectors $ v_{1}=\left( 1 , \frac{2-3\sqrt{2}}{7} \right)$ and $ v_{2}=(0,1) ,$ respectively.
\item $P^{\ast}$ is a global repeller: for every neighborhood $ W \subset D $ of $ P^{\ast} $ it holds that 
\begin{equation*}
\bigcap_{n\in \mathbb{N}}H^{-n}(W)=\left\{ P^{\ast}\right\}
\end{equation*} 
In fact, there exists a family of curves $\mathcal{F} =\left\{\eta_{s}: \; s\in \mathbb{R} \cup\left\{ \infty \right\} \right\} $ such that
	\begin{enumerate}
	\item For every $ (a,b)\in D $ there exists $ s \in \mathbb{R} \cup \left\{ \infty \right\} $ such that $ (a,b)\in \eta_{s} .$
	\item For every $ s \in \mathbb{R} \cup \left\{ \infty \right\} $ it holds that $ H^{-1}(\eta_{s}) \subset \eta_{s} .$
	\item If $ s_{1} \neq s_{2},$ then $ \eta_{s_{1}} \cap \eta_{s_{2}} = \left\{ P^{\ast} \right\} .$ A tangent vector of $ \eta_{s} $ at $ P^{\ast} $ is $ v_{1} $ if $ s\in \mathbb{R} $ and $ v_{2} $ if $ s=\infty .$
	\end{enumerate}
\end{enumerate}
\end{proposition}

\begin{proof}
The first statement is immediate from the definition of $H$ and the fact that 
\begin{equation*}
\frac{\partial H_{2}}{\partial b}(a,b)=\frac{a^{2}+1}{a^{2}(1+a-ab)^{2}}>0,
\end{equation*}
where $ H_{2} $ stands for the second component of $ H .$

First statement, particularly, implies that $H$ is injective. Furthermore, $H$ is differentiable with
\begin{equation*}
det(DH(a,b))=\frac{4a(a^{2}+1)}{(1+a-ab)^{2}}>4,
\end{equation*}
for every $ (a,b) \in D. $ Thus, $ H $ is an area expanding diffeomorphism. The rest of the claims in the second statement are easily obtained from  the fact that
\begin{equation*} 
DH(1,\sqrt{2}) =
\left(
\begin{array}{cc}
4 & 0 \\ 
2-\sqrt{2} & 3+2\sqrt{2}
\end{array} 
\right) .
\end{equation*} 

To prove the last statement we begin by observing that, from the definition of $H,$ it is clear that $\bigcap_{n\in \mathbb{N}}H^{-n}(W)$ is contained in the fiber $V_{1}$. The restriction of $H^{-1}$ to this fiber is a contractive map given by $H^{-1}(1,b)= \left(1,\frac{2(b+1)}{2+b} \right)$ whose fixed point is $P^{\ast}$. Therefore, $\bigcap_{n\in \mathbb{N}}H^{-n}(W)=\left\{ P^{\ast}\right\}.$

In order to obtain the family of curves $\mathcal{F}$ announced in the third statement we use the Linearization Theorem of Sternberg \cite{Ste}. According to this result, there exist a neighborhood $ U $ of $ P^{\ast},$ a neighborhood $ V $ of $\mathcal{O} $ and a $ C^{\infty} $-diffeomorphism $ h : U \rightarrow V $ such that
\begin{equation*}
h \circ H(a,b) = L \circ h(a,b)
\end{equation*} 
for every $ (a,b) \in H^{-1}(U),$ where $ L $ is the linear map $ L(x,y) = (\lambda_{1}x,\lambda_{2}x) .$

The family $ \mathcal{F}^{\prime} $ of curves $ \eta^{\prime}_{s} $ given by $ y = sx^{\nu} $ with $ \nu = \frac{\log\lambda_{2}}{\log\lambda_{1}} $ and $ s \in \mathbb{R} \cup \left\{ \infty \right\} $ ($\eta^{\prime}_{\infty}$ is, by definition, the straight line $ x=0 $) satisfies
\begin{enumerate}
\item For every $ (x,y)\in \mathbb{R}^{2} $ there exists $ s \in \mathbb{R} \cup \left\{ \infty \right\}$ such that $ (x,y) \in \eta^{\prime}_{s} .$
\item For every $ s \in \mathbb{R} \cup \left\{ \infty \right\} $ it holds that $ L^{-1}(\eta^{\prime}_{s}) \subset \eta^{\prime}_{s}.$
\item If $ s_{1} \neq s_{2},$ then $ \eta^{\prime}_{s_{1}} \bigcap \eta^{\prime}_{s_{2}} = \left\{ \mathcal{O} \right\} .$ A tangent vector of $ \eta^{\prime}_{s} $ at $ \mathcal{O} $ is $ (1,0) $ if $ s \in \mathbb{R} $ and $ (0,1) $ if $ s=\infty .$
\end{enumerate}
Since there exists $ n \in \mathbb{N} $ such that $ H^{-n}(D) \subset U ,$ we complete the proof by taking the family $\mathcal{F} $ formed by the curves $ \eta_{s} = H^{n}(h^{-1}(\eta^{\prime}_{s})) .$
\end{proof}
 
\begin{remark}\label{remaP3Hdel}
Since $ \mathcal{P}_{3} = \left\{ (a,b)\in \mathcal{P} : \; a^{2} \leq \gamma_{a,b},  \; a^{3} \gamma_{a,b} \leq 2 \right\} $ it follows that
\begin{equation*}
H(\mathcal{P}_{3})=\left\{ (a,b)\in \mathcal{P} : \; 1\leq b \leq 2a^{-5/4} \right\} .
\end{equation*}
One may check that $\mathcal{P}_{3}\subset H(\mathcal{P}_{3})$, see Figure \ref{figHdelP3}. Hence, there exists a chain 
\begin{equation}\label{chain}
A_{0} \varsupsetneq A_{1} \varsupsetneq A_{2} \varsupsetneq \dots \varsupsetneq A_{n} \varsupsetneq \dots 
\end{equation}
of subsets of $\mathcal{P}$ such that:
\begin{enumerate}
\item $A_{0}=H(\mathcal{P}_{3})$ and $A_{1}=\mathcal{P}_{3}$. 
\item $H(A_{n})=A_{n-1}$ for every $ n \in \mathbb{N} .$
\item $\bigcap_{n\in \mathbb{N}}cl(A_{n})=\{ P^{\ast}\}$, where $cl(A_{n})$ denotes the clousure of $A_{n}$.
\end{enumerate}
This means that the family of \textit{EBMs} $\mathbb{F}$ introduced in \eqref{defF} can be renormalized any finite number of times, whenever the maps $ \Psi_{a,b} \in \mathbb{F} $ belong to the precise above defined set $ A_j .$
\end{remark} 

% --------------- FIGURE figHP3 ------------------------
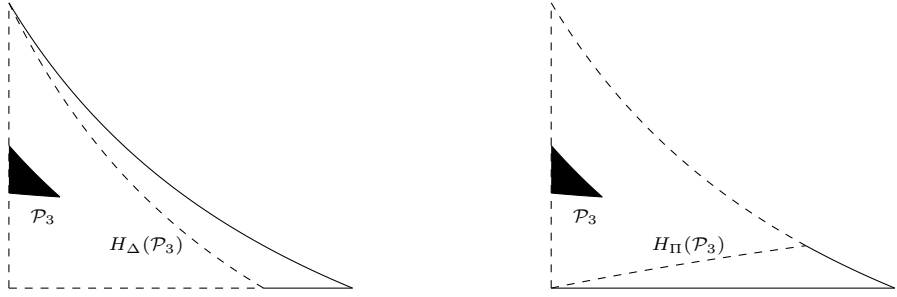
\begin{figure}[!ht]
\centering
\subfigure{\label{figHdelP3}
\begin{minipage}{0.45 \linewidth}
\centering
{\scriptsize
\begin{tikzpicture}
\begin{axis}[
			xscale=0.8,yscale=0.8,
            xmin=0.9,xmax=2.1,
            ymin=0.9,ymax=2.1,
            axis x line=none,
            axis y line=none
            ]

    % EJES
             
	\addplot[black,dashed] coordinates {(1,1) (1,2)};
%	\addplot[black] coordinates {(1,1) (2,1)};
	\addplot[black] coordinates {(2^(4/5),1) (2,1)};
            
    % CURVAS
            
    \addplot [domain=1:2,samples=50]({x},{(2/x)}); 
    \addplot [name path=A,domain=1:2^(1/5),samples=50]({x},{(2+x^2+x^3)/(1+x+x^3)});  
    \addplot [name path=B,domain=1:2^(1/5),samples=50]({x},{2*(1+x+x^3)/(x*(2+x^2+x^3))});
    
    \addplot [black,dashed,domain=1:2^(1/5),samples=50]({x^4},{1}); % H_{\Delta}
    \addplot [black,dashed,domain=1:2^(1/5),samples=50]({x^4},{2*x^(-5)}); % H_{\Delta}

	\addplot fill between[of=A and B];
    	
%	\addplot[black,dashed,mark=*,mark options={scale=0.7}] coordinates {(1.3820,1.0000) (1.3820, 1.4480)};

	% ETIQUETAS
				
%	\fill[black] (axis cs:1.7411,1.0000) circle (0.5mm); 			
	\node[black] at (axis cs:1.1000,1.2500) {$\mathcal{P}_3$}; 
	\node[black] at (axis cs:1.4000,1.1500) {$H_{\Delta}(\mathcal{P}_3)$}; 
		
\end{axis}
\end{tikzpicture}
}
\end{minipage}
}
\subfigure{\label{figHpiP3}
\begin{minipage}{0.45 \linewidth}
\centering
{\scriptsize
\begin{tikzpicture}
\begin{axis}[
			xscale=0.8,yscale=0.8,
            xmin=0.9,xmax=2.1,
            ymin=0.9,ymax=2.1,
            axis x line=none,
            axis y line=none
            ]

    % EJES
             
	\addplot[black,dashed] coordinates {(1,1) (1,2)};
	\addplot[black] coordinates {(1,1) (2,1)};
            
    % CURVAS
            
    \addplot [domain=2^(4/5):2,samples=50]({x},{(2/x)}); 
    \addplot [name path=A,domain=1:2^(1/5),samples=50]({x},{(2+x^2+x^3)/(1+x+x^3)});  
    \addplot [name path=B,domain=1:2^(1/5),samples=50]({x},{2*(1+x+x^3)/(x*(2+x^2+x^3))});
    
	\addplot [black,dashed,domain=1:2^(1/5),samples=50]({x^4},{x}); % H_{\Pi}
    \addplot [black,dashed,domain=1:2^(1/5),samples=50]({x^4},{2*x^(-4)}); % H_{\Pi}
    	
	\addplot fill between[of=A and B];

	% ETIQUETAS
				
%	\fill[black] (axis cs:1.7411,1.0000) circle (0.5mm); 			
	\node[black] at (axis cs:1.1000,1.2500) {$\mathcal{P}_3$}; 
	\node[black] at (axis cs:1.4000,1.1500) {$H_{\Pi}(\mathcal{P}_3)$}; 

\end{axis}
\end{tikzpicture}
}
\end{minipage}
}
\caption{(a) Filled in black, the set $ \mathcal{P}_3 ;$ encircled in a dashed black line, the set $ H_{\Delta}(\mathcal{P}_3) .$ (b) Filled in black, the set $ \mathcal{P}_3 ;$ encircled in a dashed black line, the set $ H_{\Pi}(\mathcal{P}_3) .$}
\end{figure}

%-------------------------------------------------

Similar arguments used to prove Proposition \ref{proHdel} may be used to obtain a similar result for $H=H_{\Pi}.$

\begin{proposition}\label{proHpi}
For $H=H_{\Pi}$ it holds that:
\begin{enumerate}
\item Each vertical fiber $ V_{\alpha} = \left\{(a,b)\in D : \; a = \alpha \right\} $ is sent by $ H $ into the vertical fiber $ V_{\alpha^{4}} $ and the map $ H(\alpha,\cdot) $ is strictly increasing on $ V_{\alpha} $ for every $ \alpha>1 .$
\item $ H $ is an area expanding diffeomorphism on $ D $ having a unique fixed point $ P^{\ast}=(1,\sqrt{2}) .$ The eigenvalues of $ DH(P^{\ast}) $ are $ \lambda_{1}=4 $ and $ \lambda_{2}=3+2\sqrt{2} $ with eigenvectors $ v_{1} = \left(  1,-\frac{2+4\sqrt{2}}{7} \right) $ and $ v_{2}=(0,1) ,$ respectively.
\item $P^{\ast}$ is a global repeller: for every neighborhood $ W \subset D $ of $ P^{\ast} $ it holds that 
\begin{equation*}
\bigcap_{n\in \mathbb{N}}H^{-n}(W)=\left\{ P^{\ast}\right\}
\end{equation*} 
In fact, there exists a family of curves $\mathcal{F} =\left\{\eta_{s}: \; s\in \mathbb{R} \cup\left\{ \infty \right\} \right\} $ such that
	\begin{enumerate}
	\item For every $ (a,b)\in D $ there exists $ s \in \mathbb{R} \cup \left\{ \infty \right\} $ such that $ (a,b)\in \eta_{s} .$
	\item For every $ s \in \mathbb{R} \cup \left\{ \infty \right\} $ it holds that $ H^{-1}(\eta_{s}) \subset \eta_{s} .$
	\item If $ s_{1} \neq s_{2},$ then $ \eta_{s_{1}} \cap \eta_{s_{2}} = \left\{ P^{\ast} \right\} .$ A tangent vector of $ \eta_{s} $ at $ P^{\ast} $ is $ v_{1} $ if $ s\in \mathbb{R} $ and $ v_{2} $ if $ s=\infty .$
	\end{enumerate}
\end{enumerate}
\end{proposition}

\begin{remark}\label{remaP3Hpi}
For $H=H_{\Pi}$ it follows that
\begin{equation*}
H(\mathcal{P}_{3})=\left\{ (a,b)\in \mathcal{P} : \; ab \leq 2 , \; b \geq a^{1/4}\right\} .
\end{equation*}
Then one may check that $\mathcal{P}_{3} \subset H(\mathcal{P}_{3}),$ see Figure \ref{figHpiP3}, and hence a chain like the one given in \eqref{chain} may be defined. Therefore,the family $ \mathbb{F} $ of \textit{EBM} can be also renormalized any finite number of times on the restrictive domain $\Pi.$
\end{remark}

\section{Proof of Theorem \ref{tmab}}\label{sectmab}
Let us consider the renormalization operator $H=H_{\Delta}$ given in \eqref{defHdel}. First statement in Theorem \ref{tmab} easily follows from Remark \ref{remaP3Hdel}. In fact, it suffices to take into account that the curve of parameters, see also \eqref{defgam0}, given by 
\begin{equation*}
\gamma_0=\left\{(16t^8,\frac{1}{2t^3}): \; t\in\left[\frac{1}{\sqrt{2}},\frac{1}{\sqrt[5]{4}}\right] \right\}
\end{equation*}
cuts each one of the sets $A_n$ defined in Remark \ref{remaP3Hdel}. Let us write
\begin{equation}\label{parametrization}
\gamma_0(t) = \left( 16t^8 , \frac{1}{2t^3} \right),
\end{equation}
for every $t\in[\frac{1}{\sqrt{2}},\frac{1}{\sqrt[5]{4}}].$ Choosing $t_n$ such that $\gamma_0(t)\in A_n,$ for every $t\in (\frac{1}{\sqrt{2}},t_n)$ the first statement holds.

In order to prove the second statement of Theorem \ref{tmab}, let us consider the invariant curve $ \eta_n $ passing through the point $ (2^{\frac{1}{2^{n+1}}},1) $ (see statement 3(a) in Proposition \ref{proHdel}).

Let us observe that the curve of parameters $\gamma_0$ given in \eqref{parametrization} may be written as
\begin{equation}\label{funciongamma}
\gamma_0=\{(a,\varphi_0(a)): \; a\in[1,a_0]\}   
\end{equation}
with $\varphi_0(a)=\sqrt{2}a^{-\frac{3}{8}}$ and $a_0=2^\frac{4}{5}.$ 

From statement 3(c) in Proposition \ref{proHdel}, the slope at $ P^{\ast} $ of any curve $ \eta_n $ is 
\begin{equation}\label{defpte}
\alpha=\frac{2-3\sqrt{2}}{7} \; .
\end{equation} 
Since $\varphi_{0}^{\prime}(1)=-\frac{3\sqrt{2}}{8}<\alpha,$ it is clear that for each natural number $ n $ the curve $ \gamma_0 $ intersects $\eta_n$ in a first point $\gamma_0(t_n^{\ast})\neq P^{\ast}.$ This means that, there exists some minimum value of $t,$ denoted by $t_n^{\ast}$ such that $\gamma_0(t_n^{\ast})\in \eta_n.$ Let $k\in \mathbb{N}$ such that
$\gamma_0(t_n^{\ast})\in A_k\setminus A_{k+1}.$ Then, there exists some value of $t,$ denoted by $\overline{t}_n$ such that $\overline{t}_n \in (\frac{1}{\sqrt{2}}, t_n^{\ast})$ and $H^{k+1}(\gamma_0(\overline{t}_n))=(a_n,1),$ with $a_n\leq 2^{\frac{1}{2^{n+1}}},$ see Figure \ref{figfinal}. This implies that there exists some interval of parameters $I_n=(\hat{t}_n,\overline{t}_n)$ such that, for every $t\in I_n$ one has $H^{k+1}(\gamma_0(t))\in \mathcal{P}_{1,m+1} $ with $ m\geq n,$ see \eqref{defP1n}. Therefore, $ \Lambda_t $ exhibits, according to Proposition \ref{procoedel}, $ 2^m\geq 2^n $ strange attractors whenever $ t \in I_n .$

% --------------- FIGURE onedimtentmapren ------------------------
\begin{figure}[!ht]
\centering
\begin{minipage}{0.9 \linewidth}
\centering
{\small
\begin{tikzpicture}
\draw (0,5) -- (0,0) -- (5,0);
\draw  plot[smooth, tension=.7] coordinates {(0,4.5)};
\draw  plot[smooth, tension=.7] coordinates {(0,4.5) (1,4) (5,3.5)};
\draw (0,4.5) .. controls (0,4.6) and (1,4) .. (4.2,-0.5);
\draw (0,4.5) .. controls (0,4.5) and (0.5,1) .. (4.2,-0.5);

\node [black] at (0,5.25) {$ a=1 $}; 
\node [black] at (5.5,0) {$ b=1 $}; 

\draw[->]  (0,4.5) -- (2,4.25);
\node [black] at (2.0,4.5) {$ v_1 $}; 

\filldraw [black] (0,4.5) circle (1pt);
\node [black] at (-0.25,4.5) {$ P^{\ast} $}; 

\filldraw [black] (0.5,4.15) circle (1pt);
\node [black] at (1.15,5) {$ \gamma_0(t_n^{\ast}) $}; 
\draw[->]  (1,4.75) -- (0.6,4.25);

\filldraw [black] (2.35,2) circle (1pt);
\node [black] at (3.3,2) {$ H^{k}(\gamma_0(t_n^{\ast}))$}; 

\filldraw [black] (4.2,-0.5) circle (1pt);
\node [black] at (4.2,-0.8) {$ H^{k+1}(\gamma_0(t_n^{\ast}))$}; 

\node [black] at (4.25,3.75) {$ \gamma_0 $}; 
\node [black] at (3.5,1) {$ \eta_n $};
\node [black] at (0.90,1) {$ H^{k+1}(\gamma_0) $}; 

\filldraw [black] (3.2,0) circle (1pt);
\node [black] at (2.4,-0.3) {$ H^{k+1}(\gamma_0(\overline{t}_n))$};
\end{tikzpicture}
}
\caption{The relative position between $\gamma_0$ and $\eta_n.$}
\label{figfinal}
\end{minipage}
\end{figure}
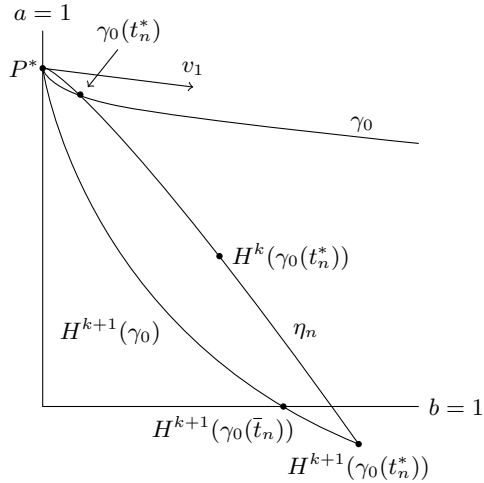
%-------------------------------------------------

\begin{remark}
We could prove Theorem \ref{tmab} by using the map $H_\Pi$ instead of $H_\Delta.$ In this case, the slope at $ P^{\ast} $ of the curves $ \eta_n $ given in statement 3(a) of Proposition \ref{proHpi} is
\begin{equation*}
\beta=-\frac{2+4\sqrt{2}}{7} \; .
\end{equation*} 
Since $\varphi_{0}^{\prime}(1)>\beta ,$ then for every $ n \in \mathbb{N} $ there exists a natural number $k$ and a sequence of values of $t,$ denoted by $\{t_n^\Pi\}$ such that $H^{k+1}(\gamma_0(t_n^\Pi))=(a_n,b_n)\in \mathcal{P}_{2,n+1},­$ with $a_nb_n=2,$ see \eqref{defP2n}. Therefore, using Proposition \ref{procoepi}, it is easy to obtain a new interval of parameters $\widetilde{I}_n$ for which $\Lambda_t$ has at least $2^n$ strange attractors whenever $t\in \widetilde{I}_n.$
\end{remark}

% ------------------------------ SECTION ------------------------------

\section*{Acknowledgements} \
This work has been supported by project MINECO-15-MTM2014-56953-P.

% ------------------------------ SECTION ------------------------------

\end{document}